\documentclass[10pt,twoside]{amsart}

\usepackage[english]{babel}
\usepackage{amsmath}
\usepackage[all]{xy}
\usepackage{amsfonts, amsthm, amssymb, amscd, amstext, amsmath, url,xr}
\usepackage{bbm}
\usepackage{stmaryrd}
\usepackage[mathscr]{euscript}
\usepackage{mathrsfs}
\usepackage{dsfont}
\usepackage{textcomp}
\usepackage{enumitem}
\usepackage[colorlinks=true,pagebackref=true]{hyperref}

\newtheorem{thm}{Theorem}[subsection]
\newtheorem{prop}[thm]{Proposition}
\newtheorem{lm}[thm]{Lemma}
\newtheorem{cor}[thm]{Corollary}
\newtheorem{thmi}{Theorem}
\newtheorem*{thmig}{Theorem}

\theoremstyle{remark}
\newtheorem{rem}[thm]{Remark}
\newtheorem{notation}[thm]{Notations}
\newtheorem{ex}[thm]{Example}

\theoremstyle{definition}
\newtheorem{df}[thm]{Definition}
\newtheorem{num}[thm]{}

\newtheorem{assumption}[thm]{Assumption}
\numberwithin{equation}{thm}

\numberwithin{equation}{thm}

\makeatletter
\DeclareRobustCommand*\cal{\@fontswitch\relax\mathcal}
\makeatother

\input cyracc.def

\DeclareMathOperator{\DM}{DM} 
\DeclareMathOperator{\DMA}{DM^A_{sm}} 
\DeclareMathOperator{\DMAT}{DM^{AT}_{sm}} 
\DeclareMathOperator{\DMATg}{DM^{\cArt T}_{sm}} 
\DeclareMathOperator{\DMAwT}{DM^{A_{tr}T}_{sm}} 
\DeclareMathOperator{\DMATS}{DM^{AT}_{tr}(\FS)} 
\DeclareMathOperator{\ATS}{AT_{tr}(\FS)}
\DeclareMathOperator{\DMATSZ}{DM^{AT}_{tr}(\FS_{\emph{Z}})}
\DeclareMathOperator{\DMATSP}{DM^{AT}_{tr}(\FS / \Phi)} 

\DeclareMathOperator{\ATSP}{AT_{tr}(\FS / \Phi)} 

\DeclareMathOperator{\ATwr}{AT_{tr}}

\DeclareMathOperator{\SH}{SH} 

\DeclareMathOperator{\Chow}{CHM} 

\newcommand{\sm}{\mathit{Sm}}
\DeclareMathOperator{\Sh}{Sh} 

\newcommand{\cArt}{\mathcal A} 

\DeclareMathOperator{\Ker}{Ker}

\newcommand{\ilim} { \varinjlim }

\newcommand{\an}{\mathrm{an}}

\newcommand{\HM}{\mathcal H_M} 

\DeclareMathOperator{\Hom}{Hom}
\DeclareMathOperator{\End}{End}

\newcommand{\Spec}[1]{\operatorname{Spec}(#1)}

\DeclareMathOperator{\car}{car}
\DeclareMathOperator{\disc}{disc}

\DeclareMathOperator{\Der}{D}
\DeclareMathOperator{\derR}{R}

\newcommand{\Rep}{\operatorname{Rep}^{\mathrm A}} 

\newcommand{\ZZ} {\mathbb Z}
\newcommand{\QQ} {\mathbb Q}

\newcommand{\CC} {\mathbb C}
\renewcommand{\AA} {\mathbb A}
\newcommand{\PP} {\mathbb P}

\newcommand{\un}{\mathbbm 1} 

\newcommand{\nis}{{\mathrm{Nis}}}
\newcommand{\et}{\mathrm{\acute{e}t}}

\newcommand{\BF}{{\mathbb{F}}}

\newcommand{\BP}{{\mathbb{P}}}
\newcommand{\BQ}{{\mathbb{Q}}}

\newcommand{\BZ}{{\mathbb{Z}}}

\newcommand{\Fg}{{\mathfrak{g}}}

\newcommand{\Fl}{{\mathfrak{l}}}

\newcommand{\FC}{{\mathfrak{C}}}

\newcommand{\FS}{{\mathfrak{S}}}

\newcommand{\CA}{{\cal A}}

\newcommand{\CO}{{\cal O}}

\newcommand{\one}{\mathds{1}}

\newcommand{\rad}{\mathop{\rm rad}\nolimits} 
\newcommand{\id}{\mathop{\rm id} \nolimits}
\newcommand{\Perv}{\mathop{\rm Perv} \nolimits}
\newcommand{\red}{\mathop{\rm red} \nolimits}

\newcommand{\CH}{\mathop{\rm CH}\nolimits}

\newcommand{\q}[1]{``#1''}
  
\newdir^{ (}{{}*!/-3pt/\dir^{(}}
\newdir_{ (}{{}*!/-3pt/\dir_{(}}
\newdir^{ )}{{}*!/+3pt/\dir^{)}}
\newdir_{ )}{{}*!/+3pt/\dir_{)}}

\def\CHM{\mathop{\rm CHM}\nolimits}
\def\HM{\mathop{\rm HM}\nolimits}
\def\Corr{\mathop{\rm Corr}\nolimits}
\def\Rad{\mathop{\rm Rad}\nolimits}
\def\Fl{\mathop{\rm Fl}\nolimits}
\def\Hom{\mathop{\rm Hom}\nolimits}
\def\Sing{\mathop{\rm Sing}\nolimits}
\def\dim{\mathop{\rm dim}\nolimits}
\def\ker{\mathop{\rm ker}\nolimits}
\def\corank{\mathop{\rm corank}\nolimits}
\def\Qbar{{\overline Q}}
\def\Xbar{{\overline X}}
\def\Vbar{{\overline V}}
\def\Ebar{{\overline E}}
\def\qbar{{\overline q}}

\author{M. Cavicchi}
\email{cavicchi@unistra.fr}
\address{IRMA (UMR7501), 7 rue Ren\'e-Descartes 67084 Strasbourg Cedex; France}

\author{F.~D{\'e}glise} 
\email{frederic.deglise@ens-lyon.fr}
\address{CNRS, IMB (UMR5569), ENS de Lyon, UMPA; 46 all\'ee d'Italie 69364 Lyon Cedex 07; France}

\author{J.~Nagel}
\email{johannes.nagel@u-bourgogne.fr}
\address{IMB (UMR5584), CNRS, Universit\'e Bourgogne-Franche-Comt\'e; 9 avenue Alain Savary 21000 Dijon Cedex; France}

\begin{document}

\title[Motivic decompositions of families with Tate fibers]{Motivic decompositions of families with Tate fibers: smooth and singular cases}


\begin{abstract}
We apply Wildeshaus's theory of motivic intermediate extensions to the motivic decomposition conjecture, formulated by Deninger-Murre and Corti-Hanamura.
 We first obtain a general motivic decomposition for the Chow motive of an arbitrary smooth projective family $f:X \rightarrow S$
 whose geometric fibers are Tate. Using Voevodsky's motives with rational coefficients, the formula is valid for an arbitrary regular base
 $S$, without assuming the existence of a base field or even of a prime integer $\ell$ invertible on $S$. This result,
 and some of Bondarko' ideas, lead us to a generalized formulation of Corti-Hanamura's conjecture.
 Secondly we establish the existence of the motivic decomposition when $f:X \rightarrow S$ is a projective quadric bundle
 over a characteristic $0$ base, which is either sufficiently general
 or whose discriminant locus is a normal crossing divisor. This provides a motivic lift of the Bernstein-Beilinson-Deligne decomposition in this setting. 
\end{abstract}

\maketitle
\tableofcontents 

\section{Introduction}

Consider a proper, regular scheme $X$ over some base scheme $S$. A first approximation to the questions that we study in this paper may be stated as follows: if the cohomology of each geometric fiber of $X$ over $S$ admits an \q{easy} description, does the cohomology of the whole of $X$, relatively to $S$, admit such a description as well? 

In order to make this question more precise, let us start by looking at the case of a \emph{smooth} morphism $f:X \rightarrow S$. We have then the \emph{universal} cohomological invariants of each fiber $X_s$, $s \in S$, and of the scheme $X$ over $S$, at our disposal: namely, the \emph{Chow motive} $h(X_s)$ of each fiber, over the residue field $k(s)$, and the \emph{relative Chow motive}\footnote{Chow motives are meant to be with rational coefficients. For a recall of the definition of the category of Chow motives over a general base scheme, see Section \ref{sec:weight}.} $h_S(X)$. The first notion of \q{easy} description of $h(X_s)$ which comes to mind is then provided by the requirement that $h(X_s)$ be a \emph{Tate motive}, i.e. isomorphic to a finite direct sum of twists $\one(n_i)$ of the unit object, where $\{ n_i \}_{i \in I}$ is a finite collection of integers. 

When $S$ is the spectrum of an algebraically closed field $k$, both the motive $h_s(X)$ of the unique geometric fiber and the relative motive $h_S(X)$ coincide of course with the Chow motive $h(X)$ of $X$. Fix some Weil cohomology theory $H^{*}( \cdot )$ on the category of smooth projective varieties over $k$. The following result provides a criterion for $h(X)$ to be Tate, under additional assumptions: 
\begin{thmig}[Kimura, Vial]
 Assume $k$ is an algebraically closed field of characteristic $0$.
 Suppose that $X$ is a smooth projective variety over $k$ such that the Chow motive $h(X)$ is \emph{finite dimensional} (in the sense of Kimura) and such that $H^*(X)$ is algebraic (i.e, the rational cycle class maps $\CH^i(X)_{\QQ}\to H^{2i}(X,\QQ)$ are surjective for all $i$.). Then 
$$
h(X) = \bigoplus_i \QQ(-i)^{b_{2i}}
$$
is Tate.
\end{thmig}
 This result generalises a theorem of Jannsen and was observed by Vial \cite{Vial:projectors} and Kimura \cite{Kimura} (in the case where one replaces  $h(X)$ by an arbitrary Chow motive $M$). A proof will be sketched in Section \ref{sec:dec}. 
 
Our first theorem may be seen as the natural generalization of the previous one to the relative case. A first result in this direction was obtained by Vial in \cite[Thm. 4.2]{Vial:fibrations}. It states that if $f:X\to S$ is a flat morphism between $k$-varieties such that
 $\CH^i(X_s)\otimes\QQ\cong\QQ$ for all $i$ and for all $s\in S(\Omega)$
 (where $\Omega$ is a universal domain over $k$) then the relative motive $h_S(X)$ is a Tate motive.
 Our theorem extends this result in two directions. First we consider geometric fibers: this implies
 that we will get more general motives than Tate motives.
 Secondly, we work over an arbitrary regular base scheme $S$,
 with no assumption on the existence of a base field or of a prime $\ell$ invertible on $S$.
  In order to give the statement, recall that a representation of a group $G$ (say, on a finite dimensional $\BQ$-vector space) is \emph{Artin} if it factors through a finite quotient of $G$.
\begin{thmi}[see Theorem \ref{thm:decomposition1}]\label{thmi:decomposition1}
Let $S$ be a regular connected scheme with \'etale fundamental group $\pi$. 
Let $f:X \rightarrow S$ be a smooth proper morphism whose geometric fibers are Tate
 (\emph{i.e.} the associated Chow motives are Tate).
Then there exists an isomorphism of motives over $S$:
$$
h_S(X) \simeq \bigoplus_{i \in I} \rho_!(V_i)(-n_i)[-2n_i]
$$
where $I$ is a finite set, $V_i$ is a simple Artin $\QQ$-representation of $\pi$,
 $n_i$ is a non-negative integer and $\rho_!(V_i)$
 is the rational motive naturally associated with $V_i$.

Moreover, the set $\big\{\big([V_i],n_i\big), i \in I \big\}$ where $[V_i]$
 denotes the isomorphism class in the category of Artin representations $\Rep_\QQ(\pi)$
 of $\pi$  is uniquely determined. 
\end{thmi}

The property which characterizes the set $\big\{\big([V_i],n_i\big), i \in I \big\}$ will be formulated in a precise way in the complete statement of Theorem A, see Theorem \ref{thm:decomposition1}. For the moment being, let us just point out that Theorem A shows that the naive generalisation of the theorem of Kimura-Vial to the relative situation is false:
 if $f:X\to S$ is a smooth, projective morphism such that $h(X_s)$ is  finite dimensional and such that $H^*(X_s)$ is algebraic for all $s\in S$,
 the relative motive $h_S(X)$ does not necessarily decompose as a direct sum
 of relative Lefschetz motives\footnote{This statement is even false if one assumes that the fibers are homogeneous spaces as in [Iyer]},
 since it is an object of a more general kind - an \emph{Artin-Tate} motive. This can be seen already on the topological level.
 If $f:X\to S$ is a smooth, projective morphism whose fibers have algebraic cohomology then the direct image sheaves $R^if_*\QQ$ are local systems with finite monodromy,
 but they are not necessarily constant. A concrete example is given by smooth quadric fibrations of even relative dimension (Corollary \ref{cor:sm_quadrics}).
 We have recorded in Example \ref{arith_smprquad} some constructions of smooth proper quadrics over $\Spec \BZ$,
 and more generally over rings of integers of number fields, to which our result applies; for rings of integers whose spectrum is non-simply connected,
 we give explicit cases where relative Artin-Tate motives with non-trivial monodromy actually appear.
 
\medskip{\bf The singular case.}
We want now to treat the general case of a regular scheme $X$, proper over some regular base $S$. In order to give a meaning to the motive of the \emph{singular} fibers of $f:X \rightarrow S$, one has to work in the framework of constructible \emph{rational mixed motives} $\DM_c(S)$, satisfying the six functors formalism and Grothendieck-Verdier duality. This formalism is by now available after the seminal work of Voevodsky, finalized by Ayoub and Cisinski-D\'eglise. The notion of Tate motive extends to this setting. By asking $f$ to have Tate fibers, one may still hope to obtain a nice decomposition of the motive $h_S(X)$ as in Theorem A. We will explain in Section \ref{sec:dec} that the conjectural existence of a motivic t-structure on $\DM_c(S)$ and some ideas of Bondarko give ground to this belief. One should in fact expect the existence of some suitable stratification $\{ S_{\phi} \}$ of $S$, with one stratum being the open locus $U$ of smoothness of $f$, and the pullbacks $X_{S_{\phi}}$ to the strata of $S \setminus U$ having fixed degeneration type, in some suitable sense. The desired decomposition of $h_S(X)$ should be then expressed in terms of \q{easy} objects lying over each $S_{\phi}$. 

In order to get an explicit hold on the possible degenerations of the fibers, we restrict our attention to the special situation of quadric bundles in characteristic zero. Our Theorem B then shows that in at least two reasonable cases (one of them being the \q{generic} one), the previous hypothetical picture holds indeed true. 

\begin{thmi}[see Cor. \ref{quadricCH}, Cor. \ref{quadricCHsncd}] \label{thmi:decomposition2}
Let $S$ be a scheme over a field of characteristic $0$,
 and $f:X\to S$ be a projective quadric bundle such that one of the following conditions holds:
\begin{enumerate}
\item the discriminant of $X/S$ is a normal crossing divisor in $X$;
\item $X/S$ is \emph{regular} as a quadric bundle (which means "generic" in a precise sense, see Def. \ref{def:regular}).
\end{enumerate}
We consider the stratification $\{S_{\phi}\}$ of $S$ given respectively by 
\begin{enumerate}
\item the intersections of the irreducible components of the discriminant (see Def. \ref{strat_sncd})
 in case (1);
\item corank (see after Def. \ref{def:regular}) in case (2).
\end{enumerate}
Then there exists a finite set of simple Artin representations $\{ V_{i,\phi} \}$, whose isomorphism classes are uniquely determined, such that $h_S(X)$ admits a motivic decomposition
$$
h_S(X)\simeq\bigoplus_{i,\phi} j_{\phi!*}(\rho_!(V_{i,\phi})(-n_{i,\phi}))[-2n_{i,\phi}]
$$
with $j_{\phi!*}( \cdot )$ denoting a certain extension of the given motive to a Chow motive on the closure $\overline{S_{\phi}}$, unique up to isomorphism. 
\end{thmi}

Once again, we are leaving to the main body of the paper the precise formulation of the property characterizing the extensions $j_{\phi!*}(\rho_!(V_{i,\phi})(-n_{i,\phi}))$ up to isomorphism. We want however to stress that the choice of the notation is not an accident: these objects are precisely obtained by defining and applying an \emph{intermediate extension} functor $j_{\phi!*}$ at the level of motives, modeled on the intermediate extension of perverse sheaves introduced in \cite{BBD82}. As a matter of fact, Theorem B plays precisely the role of a motivic lifting of the celebrated decomposition theorem of \emph{loc. cit.}

The existence of such motivic decompositions, in the presence of singularities of $f$,
 had been previously obtained only in a few cases. First by Gordon, Hanamura, Murre in \cite[Th. I in \textsection 2]{GHM}
 and then by Nagel and Saito in \cite{NS}, for conic bundles over a surface (see also \cite{Bouali}). The existence of motivic intermediate extensions is especially important in the motivic part of the Langlands program, where the first construction of a Chow motive realizing to intersection cohomology was provided by Scholl in \cite{Scholl}, in the case of modular curves. In recent years, Wildeshaus has started a systematic study of motivic intermediate extensions based on the formalism of $\DM(S)$ and Bondarko's weight structures. He and his students successfully applied these methods to many families of higher dimensional Shimura varieties, the most recent one being the case of \emph{genus two Hilbert-Siegel varieties} treated by the first author in \cite{Cav19}. The most general construction applies to some relative motives of "abelian type" (see \cite{Wil17}).

Wildeshaus's techniques have never been applied to the motivic decomposition problem.  This is what we do in our proof of Theorem B and indeed, the desire t obtain such application was the original motivation of our paper.

Let us conclude by observing that both our main Theorems A and B may be seen as expressing the existence of interesting \emph{algebraic cycles}, namely, the relative correspondences on $X \times_S X$ acting on $h_S(X)$ as projectors on the direct factors of our decompositions. From this viewpoint, our results provide a positive answer to the relative analogue of a well-known conjecture dealing with smooth projective varieties over a field, the \emph{Chow--K\"unneth conjecture} - itself a strengthening of the \emph{K\"unneth conjecture}, one of the famous standard conjectures on algebraic cycles. We refer the reader to Section \ref{sec:dec} for an extended discussion of this conjecture, of its relative counterparts and of the various notions of motivic decomposition that we consider in this work.

\bigskip

\noindent \textbf{Possible developments.} We believe that our method for proving Theorem \ref{thmi:decomposition2},
 based in particular on Theorem \ref{thmi:decomposition1}, should apply to other situations.
 Firstly one could investigate projective families of a different type, 
 for example degenerations of Grassmannians or Severi-Brauer varieties.
 The geometry of hyper-k\"ahler varieties should also offer a suitable playground.
 Lastly, the Example \ref{arith_smprquad} suggests that one could look at the arithmetic case,
 where we work with (flat regular) schemes over rings of integers. In particular,
 the theory of \emph{regular} quadric fibrations in this case still seems manageable according to \cite{APS}.

\section*{Plan of the paper}

In Section 2, we recall the theory of gluing and weights for rational mixed motives,
and state several versions of \emph{motivic decompositions}, in order of generality.
 We then give some basic definitions and computations of relative \emph{smooth} Artin motives,
 analog of Artin motives over a field. When $\pi$ is the fundamental group of a regular base $S$, we show an equivalence of categories between the bounded derived category of Artin $\pi$-representations and the full subcategory of constructible motives formed by smooth Artin-Tate motives $\DM_c(S)$ (Prop. \ref{prop:Artin_motives}), extending a classical result of Voevodsky-Orgogozo (the case when $S$ is the spectrum of a field). Finally, we deduce our first motivic decomposition theorem, by combining a "spreading-out" technique valid for constructible motives, the existence of weights on smooth Artin-Tate motives (see in particular Thm. \ref{thm:Artin-Tate}), and an extension argument based on Wildeshaus's analysis of weights and of minimal/intermediate extensions of
 pure motives, relying on a new notion of \emph{fair} extension (Def. \ref{df:fair_ext}). 

Section 3 deals with an extension of Wildeshaus's theory of motivic intermediate extensions
 in order to be able to deal with (smooth) Artin-Tate motives and their extensions along immersions. Recall that Wildeshaus's construction,
 given a stratification $\Phi=\{S_{\phi}\}$ of the base $S$,
 consists in constructing a glued subcategory of Chow motives starting from
 constant local systems upstairs. Based on a criterion of \emph{semi-primality}
 (see Def. \ref{semipr}), this glueing procedure guarantees
 the existence of the functors $j_{!*}$ when $j$ is the immersion between (union of) strata.
 We have two difficulties to circumvent in order to apply this construction.

The first one is that we need to consider more general local systems, as it is known
 from the work of Beauville that in the case of an algebraically closed base field,
 \emph{Prym motives} should appear in the motives of quadric bundles. As these correspond
 to an \'etale cover of degree $2$,\footnote{given by the spectrum of the associated Clifford algebra
 see \ref{num:clifford_aglebra&Prym}, or by the Stein factorization of the Fano scheme
 see the proof of Theorem \ref{quadric_AT}} we are led to consider smooth Artin-Tate motives 
 as natural \emph{motivic} local systems. But then, in order to adapt Wildeshaus's construction,
 one needs to control the degeneracy of smooth Artin-Tate motives,
 as was already done in the smooth case, in the proof of Theorem \ref{thmi:decomposition1}.
 Fortunately, by \emph{purity of the branch locus} we need only to control the degeneracy
 in codimension $1$. Then it is sufficient to assume that the smooth Artin-Tate motives are
 tamely ramified to ensure sufficient regularity of the degenerations; see the proof of
 Theorem \ref{constmot}.

The second difficulty is that to perform the glueing construction,
 Wildeshaus has to assume that the closures of strata of $\{S_{\phi}\}$ are smooth.
 This assumption is satisfied in the case (1) of Theorem \ref{thmi:decomposition2}, but not in the case (2) of regular quadric bundles. In order to deal with these singularities, we have to consider
 a suitable resolution of the strata, $\mathfrak{S} \rightarrow \Phi$, in such a way that we can control
 the pushforward of the tamely ramified smooth Artin-Tate motives appearing on the
 (abstract) blow-ups. Then Wildeshaus's glueing construction goes through and one obtains
 a suitable category of tamely ramified $(\mathfrak S/\Phi)$-constructible
 Artin-Tate motives $\DMATSP$ (see Definition \ref{imstrmot}).
 Moreover, this category satisfies the following good properties:
\begin{itemize}
\item Bondarko's weight structure restricts to this subcategory.
 The corresponding weight $0$ part, denoted by $\ATSP$,
 is a pseudo-abelian subcategory of $\CHM(S)$ which is semi-primary
 (Theorem \ref{artatesp}).
\item The functor $j_{!*}$ exists, for $j$ an inclusion of strata ,
 and commutes with the realization functors (Theorem \ref{real}).
\item The realization functors restricted to $\DMATSP$ are conservative
 (Theorem \ref{cons}).
\end{itemize}

Section 4 deals with our last theorem, the motivic decomposition of certain families
 of projective quadric bundles. We first explain our notion of regularity for quadric bundles
 (Definition \ref{def:regular}) and exploit it to construct the required resolution $\mathfrak{S} \rightarrow \Phi$. This is achieved by considering the geometry of quadrics, more precisely by using the flag variety associated with the quadratic vector bundle of the regular quadric bundle. See
 Proposition \ref{prop: regular} and Lemma \ref{lemma: Assumption2}. We then show that the motive of such quadric bundles belongs to a category obtained by gluing tamely ramified Artin-Tate motives, and deduce from the results of Section 3 our first motivic decomposition (Corollary \ref{cor: decomposition}).
 The last part of Section 3 deals with the motivic decomposition in the case
 where the discriminant locus is a normal crossing divisor.

\section*{Acknowledgements}

The authors want to thank Emiliano Ambrosi, Giuseppe Ancona, Daniele Faenzi, Sophie Morel, J\"org Wildeshaus for discussions, ideas and interest during the course of writing of this paper.
 We also heartily thank the referee for his useful comments which helped us to improve our text.
 This work received support from the ANR HQDIAG (contract ANR-21-CE40-0015),
 the French "Investissements d'Avenir" program, project ISITE-BFC (contract ANR-lS-IDEX-OOOB), and the ANR Hodgefun (contract ANR-16-CE40-0011).

\section*{Notations and conventions}\label{sec:notations}

Given an arbitrary group $G$, and a field $K$ of characteristic $0$,
 we will denote by $\Rep_K(G)$ the abelian monoidal
 category of finite dimensional $K$-representations $V$ of $G$ which factor
 through a finite quotient of $G$. When $G$ is a pro-finite group,
 this amounts to the continuity of the action.
 When $G=\pi$ is the \'etale fundamental group of a geometrically pointed scheme $X$,
 or the usual fundamental group of a complex variety, representations in $\Rep_K(\pi)$
 are classically called \emph{Artin representations}.
 In any case, thanks to Maschke's lemma, the category  $\Rep_K(G)$ is semi-simple.

\bigskip

All our schemes are implicitly assumed to be excellent noetherian and finite dimensional.

 Given (such) a scheme $S$,
 we denote by $\DM_c(S,\QQ)$ the triangulated category of constructible rational
 motives over $S$.\footnote{In \cite{CD3}, several models of this category are given:
 Beilinson motives (Def. 15.1.1), Voevodsky's h-motives (Th. 16.1.2), Voevodsky's motivic complexes
 ($S$ geometrically unibranch scheme, Th. 16.1.4), the plus-part of the rational stable homotopy category (Th. 16.2.13 and 5.3.35),
 the $\PP^1$-stable $\AA^1$-derived \'etale category (Th. 16.2.22). All models being equivalent, the reader
 is free to choose his preferred one.}
 We let $\un_S$ be the constant motive over $S$
 (which is also the unit of the monoidal structure).
 We will use \emph{cohomological motives} over $S$:
 given any morphism $f:X \rightarrow S$, we put:
$$
h_S(X)=f_*(\un_X).
$$
We will also heavily rely on the Bondarko's theory of weight structure,
 and especially the canonical (Chow) weight structure on $\DM_c(S,\QQ)$.
We recall these notions in Section \ref{sec:weight}.

Another important property of rational mixed motives from \cite{CD3} that we will
 use is the so-called \emph{continuity property} (see Def. 4.3.2, Prop. 14.3.1 of \emph{op. cit.}).
 We will use it in the following form (see Prop. 4.3.4 of \emph{op. cit.}):
\begin{prop}\label{prop:cont}
Let $(S_i)_{i \in I}$ be a projective system of schemes with affine transition maps,
 and $S$ be its projective limit $S$, assumed to exists in our category of schemes. 
 Then the canonical functor:
$$
2-\ilim_{i \in I} \DM_c(S_i,\QQ) \rightarrow \DM_c(S,\QQ)
$$
is an equivalence of categories.
\end{prop}

We will use two kinds of triangulated realization functors:
\begin{itemize}
\item Let $\ell$ be a prime invertible on $S$.
 We have the $\ell$-adic realization:
$$
\rho_\ell:\DM_c(S,\QQ) \rightarrow \Der^b_c(S_\et,\QQ_\ell)
$$
where the right-hand side is the constructible derived category
 of Ekedahl's $\ell$-adic \'etale sheaves with rational coefficients
 (see \cite[7.2.24]{CD4}).
\item Let $E$ be a characteristic $0$ field given
 with a complex embedding $\sigma:E \rightarrow \CC$.
 Assume $S$ is a finite type $E$-scheme.
 We have the Betti realization:
$$
\rho_B:\DM_c(S,\QQ) \rightarrow \Der_c(S^\an,\QQ)
$$
where the right-hand side is the constructible derived category
 of rational sheaves over the analytical site of $S^\an=S^\sigma(\CC)$.
 This realization is obtained from that of \cite{AyoubBetti} as the following
 composite:
$$
\DM_c(S,\QQ) \simeq \SH_c(S)_{\QQ+}
 \subset \SH_c(S)_\QQ \xrightarrow{\mathrm{Betti}_S \otimes \QQ} \Der_c(S^\an,\QQ)
$$
where:
\begin{itemize}
\item $\SH_c(S)$ is the constructible stable $\AA^1$-homotopy category
 (made of compact spectra over $S$).
\item $\SH_c(S)_{\QQ+}$ is plus-part of the rationalization of $\SH_c(S)$
 (see eg. \cite[16.2.1]{CD3})
\item the first equivalence is given by \cite[Th. 16.2.13]{CD3}.
\item the functor $\mathrm{Betti}_S$ is (the obvious restriction of that) defined in \cite[Def. 2.1]{AyoubBetti}.
\end{itemize}
\end{itemize}
These two realization functors admit a right adjoint and commute with
 the six operations.\footnote{This is proved in the references indicated above.}

\section{Smooth Artin-Tate motives and motivic decompositions}
\subsection{Gluing and weights on Beilinson motives}\label{sec:weight}

\begin{num}
In the theory of rational mixed motives, a crucial property is given
 by the gluing formalism \cite[Sec. 1.4.3]{BBD82}, which is a consequence
 of Morel-Voevodsky's localization theorem \cite[Th. 2.21, p. 114]{MV}. 
 Given a closed immersion $i: Z \hookrightarrow S$
 with complementary open immersion $j: U \hookrightarrow S$,
 one has six functors
\begin{equation}\label{eq:gluing}
\xymatrix@=30pt{
\DM_c(U,\QQ)\ar@<5pt>[r]^{j_!}\ar@<-5pt>[r]_{j_*}
 & \DM_c(S,\QQ)\ar[l]|{j^*} \ar@<5pt>[r]^-{i^*}\ar@<-5pt>[r]_-{i^!}
 & \DM_c(Z,\QQ) \ar[l]|-{i_*}\, ,
}
\end{equation}
satisfying the formalism of \cite[Sec. 1.4.3]{BBD82}.\footnote{This is the so-called localization property of motives,
 extensively studied in \cite[\textsection 2.3]{CD3}.}

This property will be at the heart of our results (see in particular Section \ref{sec:semi-prim_ChowMot}).
 They are the starting point of H\'ebert's and Bondarko's extension of the weight structure on rational
 motives, from the case of perfect fields to that of arbitrary bases.
 Let us recall this theory, from \cite[Thm. 3.3, thm. 3.8 (i)-(ii)]{Heb11}, for future references.
\end{num}
\begin{thm}\label{motweightstr}
For each scheme $S$, there is a canonical weight structure $w$ on the triangulated category $\DM_c(S,\QQ)$
 called the \emph{motivic weight structure}.
 The family of these weight structures indexed by schemes $S$ is uniquely characterized by the following properties. 
\begin{enumerate}[wide, label=(\arabic*)]
\item The objects $\one_{S}(p)[2p]$ belong to the heart $\DM_c(S,\QQ)_{w=0}$ for all integers $p$, whenever $S$ is regular.
\item \label{compglu} For any morphism $f: T \rightarrow S$,
 the functor $f^*$ (resp. $f_*$) preserves negative weights (resp. positive weights).
 When $f$ is separated of finite type, the functor $f^!$ (resp. $f_!$) preserves positive weights (resp. negative weights). 
\end{enumerate}
\end{thm}
We will only use the motivic weight structure in this paper.
 So $w$ will always mean weights for the motivic weight structure.
 Over a base scheme $S$, the heart of this weight structure 
 will be denoted by $\Chow(S)$ and is called the category of \emph{Chow motives over $S$}.

\begin{rem}\label{rem:compglu} \textit{Gluing of (motivic) weights}. 
 Note in particular from point (2) that when $f$ is proper (resp. etale separated of finite type), $f_*$ (resp. $f^*$) 
 preserves weights. Moreover, under the assumption of the paragraph preceding the theorem,
 it follows from the gluing diagram \eqref{eq:gluing} and point (2) that the motivic weight structure on $S$
 "satisfies gluing": a motive $M$ on $S$ is $w$-positive (resp. $w$-negative) if and only if
 $j^*(M)$ and $i^!(M)$ (resp. $j^*(M)$ and $i^*(M)$) are $w$-positive (resp. $w$-negative).
\end{rem}

\begin{rem}\textit{Assumptions on base schemes}.
Note that in \emph{loc. cit.}, it is assumed that schemes are of finite type over an excellent base scheme $B$
 of dimension less than 3. First, this assumption is not used in the proof of the above statement.
 Secondly, the only reason this assumption appeared in motivic homotopy theory is for the proof of the Grothendieck-Verdier
 duality (also called "local duality"), \cite[Th. 4.4.21]{CD3}.\footnote{More precisely, it is used to obtain some resolution
 of singularity statement; see \emph{loc. cit.}} Since then, this result has been generalized to the schemes considered in our
 paper by Cisinski in \cite[\textsection 2.3]{CisMot}. It is used in \cite{Heb11} only in the statement of Corollary 3.9.
\end{rem}
%
%

\subsection{Motivic decompositions} \label{sec:dec}

In this section, we formulate several forms of \emph{motivic decompositions},
 which are extensions of the classical \emph{Chow-K\"unneth decomposition}. To discuss the latter, let us first recall that for a smooth projective variety $X$ over a field $k$ and a suitable Weil cohomology theory $H^*( \cdot)$, the standard conjecture $C(X)$, the K\"unneth conjecture, states that the K\"unneth components $p_i\in H^*(X\times X,\QQ))$ of the diagonal are algebraic (i.e., they are in the image of the cycle class map). 

Now the {\em nilpotency conjecture} considers the functor
$$
T:\CHM(k)\to\HM(k)
$$
from the category of Chow motives to the category of homological motives and states that if $M$ is a Chow motive and
$M_{\rm hom}$ the corresponding motive modulo homological equivalence (sometimes called the Grothendieck motive), the kernel of the ring homomorphism
$$
\End(M)\to\End(M_{\rm hom})
$$
is a nilpotent ideal. If we take $M = h(X)$, the nilpotency conjecture $N(X)$ states that the kernel of the homomorphism
$$
\Corr^0(X,X)\to\End(H^*(X,\QQ))
$$
from the ring of correspondences of degree zero to the endomorphism ring of $H^*(X)$ is a nilpotent ideal. (Note that if $X$ has pure dimension $d$, the left hand side is the Chow group $\CH^d(X\times X)_{\QQ}$, the right hand side is isomorphic to $H^{2d}(X\times X,\QQ)$ by Poincar\'e duality and the above map is just the cycle class map.)

The nilpotency conjecture has important consequences. First of all, it implies that the functor $T$ is conservative (i.e., it detects isomorphisms), hence essentially injective (nonexistence of {\em phantom motives}). Moreover, Jannsen proved that conjectures $C(X)$ and $N(X)$ imply the Chow--K\"unneth conjecture $CK(X)$, i.e., the K\"unneth projectors  lift to a set of mutually orthogonal projectors $\pi_i\in\Corr^0(X,X)$ ($i=0,\ldots,2d$) such that 
$$
\Delta_X = \sum_{i=0}^{2d} \pi_i
$$
in $\CH_d(X\times X)\otimes\QQ$ and such that $\pi_i$ depends only on the motive modulo homological equivalence. In this case the motive of $X$ admits a {\em Chow--K\"unneth decomposition}
$$
h(X) = \bigoplus_{i=0}^{2d} h^i(X)
$$
with $h^i(X) = (X,\pi_i)$. 

 \medskip
The nilpotency conjecture is known to hold for finite-dimensional motives. This means in particular that isomorphisms between finite-dimensional Chow motives can be detected by passing to the corresponding homological motives. An application is the result of Kimura-Vial stated in the introduction, which is a direct consequence of the conservativity of the homological realisation on finite-dimensional motives. \footnote{If $H^*(X)$ is algebraic there exists a finite set of algebraic cycles $\{Z_i\}_{i\in I}$, $Z_i\in\CH^{p_i}(X)$, whose cycles classes generate the cohomology. These cycles define a morphism of Chow motives 
$\alpha:\oplus_i \BZ(-p_i)\to h(X)$ such that $T(\alpha)$ is an isomorphism, hence $\alpha$ is an isomorphism by conservativity.} 

\medskip 
Let us now switch to the relative setting. In their fundamental work \cite{CH}, Corti and Hanamura defined a category $\CHM(S)$ of relative Chow motives over a quasi--projective base scheme $S$ defined over a field $k$. 
%
%
In this setup, one works with proper morphisms $f:X\to S$ with $X$ smooth and quasi--projective over $k$.
 Relative correspondences are elements of $\CH_*(X\times_S X)$ and composition of correspondences is
 defined using refined Gysin homomorphisms. If we let $Q$ denote either $\QQ$ or $\QQ_\ell$, the relative analogue of the nilpotency conjecture states that the kernel of the map
$$
\Corr^0_S(X,X)\to\End\big(Rf_*(Q_X)\big)
$$
is a nilpotent ideal. (If $X$ has dimension $d$ the above map can be identified with the cycle class map $\CH_d(X\times_S X)_{\QQ}\to H_{2d}^{\rm BM}(X\times_S X,\QQ)$ to Borel--Moore homology.) 
 Corti and Hanamura have considered 
a relative analogue of the notion of Chow--K\"unneth decomposition, by defining suitable realization functors ; nowadays, we can identify them with the restriction of either the Betti realization $\rho_B$ associated with a fixed complex embedding of $k$, or the $\ell$-adic realization $\rho_\ell$
 (see our Notations and conventions, page \pageref{sec:notations}) under the canonical fully faithful embedding of Corti-Hanamura's category into constructible motives\footnote{Indeed,
 when $S$ is quasi-projective over a perfect field,
 it is proved in \cite{Jin1} that the weight $0$ part of $\DM_c(S)$ is equivalent to the category
 defined by Corti and Hanamura.}. We will denote either of these realization functors by $\rho$.

Let us first consider the case where $f:X\to S$
 is in addition smooth. According to Deligne's theorem,
 one gets a decomposition in $D^b_c(S,Q)$:
\begin{equation}\tag{D}\label{eq:Deligne}
Rf_*(Q_X)\cong\bigoplus_i R^if_*(Q_X)[-i].
\end{equation}
One usually says that $h_S(X)$ admits a 
 {\em relative Chow--K\"unneth decomposition} 
 (in short CK-decomposition) if there exist relative motives $M_i\in\CHM(S)$ such that $h_S(X) = \bigoplus_i M_i$ and $\rho(M_i) = R f_*(Q_X)[-i]$ for all $i$. We refer the reader to \cite[Chap. 8]{MNP}
 for a thorough discussion over $\CC$.
 The first example of a relative Chow-K\"unneth decomposition was obtained
 for abelian schemes by Deninger and Murre in \cite{DenMur}.

The general case is more complicated and involves the theory of perverse sheaves.
 Recall that this theory exists either in the analytical setting,
 or in the $\ell$-adic \'etale one. So let us denote as above by $Q$
 either the field $\QQ$ or $\QQ_\ell$, intending that one uses sheaves on the associated analytical
 site in the first case, and \'etale $\ell$-adic sheaves in the second.

 When $k$ is an algebraically closed field of characteristic zero or the separable
 closure of a finite field,
 the decomposition theorem of Beilinson-Bernstein-Deligne says that a decomposition
 of the form \eqref{eq:Deligne} still exists provided one
 replaces the canonical t-structure by the perverse one.
 Moreover, each perverse cohomology sheaf $^p Rf_*(Q_X)$
 is semi-simple in the perverse heart of $D^b_c(S)$.
 These two statements can be expressed by the existence
 of a \emph{topological decomposition} in $D^b_c(S,Q)$
\begin{equation}\tag{BBD}\label{eq:BBD}
Rf_*(Q_X) = \bigoplus_{\lambda} j_{\lambda!*}(L_\lambda)
\end{equation}
where there exists an integer $n_\lambda$ such that
 $L_\lambda[n_\lambda]$ is a simple local system
 on a smooth locally closed subscheme 
 $j_\lambda:U_\lambda \subset X$ 
 and $j_{\lambda!*}(L_\lambda)$ denotes the {\em intermediate extension} (or intersection complex) defined in
 \cite{BBD82}.\footnote{\label{fn:pDeligne} Note in particular that one deduces
 isomorphisms
 ${}^pRf_*(Q_X)[-n] \simeq \bigoplus_{\lambda \mid n_\lambda=n} j_{\lambda!*}(L_\lambda)$ and:
\begin{equation}\tag{D'}\label{eq:pDeligne}
{}^pRf_*(Q_X) \simeq \oplus_n {}^pR^nf_*(Q_X)[-n].
\end{equation}
}

We can now give the first general definition of relative motivic decomposition.
\begin{df}\label{df:CH}
Let $k$ be an algebraically closed field of characteristic $0$. 
 Let $f:X \rightarrow S$ be a projective morphism
 of quasi-projective $k$-schemes such that $X$ is smooth over $k$,
 with associated BBD-decomposition \eqref{eq:BBD}.

A \emph{Corti-Hanamura decomposition} (in short: CH-decomposition)
 of the Chow motive $h_S(X)$ is a finite decomposition
$$
h_S(X) \simeq \bigoplus_{\lambda} M_\lambda
$$
in the pseudo-abelian category $\CHM(S)$ such that $\rho(M_{\lambda}) = j_{\lambda!*}(L_\lambda)$. 
\end{df}
When $f$ is smooth, the decomposition \eqref{eq:BBD} is a refinement of the decomposition \eqref{eq:Deligne}.
 So in the smooth case, a CH-decomposition is similarly a refinement of a CK-decomposition.

\begin{ex}
 Such decompositions have been constructed in a few cases,
 particularly in complex algebraic geometry: cf. \cite{GHM, dC-M, Saito-SMS}.
\end{ex}

The specificity of the above definition is to incorporate the singular case as well. The statement of our first motivic decomposition theorem,
 \ref{thm:decomposition1}, together with Bondarko's work \cite{Bon15}
 lead us to the following more general kind of motivic decompositions of pure 
 weight $0$ motives.
\begin{df} \label{df:BCH}
Let $f:X \rightarrow S$ be a proper morphism such that $X$ is regular. A {\em Bondarko-Corti-Hanamura decomposition} (in short: BCH-decomposition)
 of the Chow motive $h_S(X)$ is a finite decomposition
$$
h_S(X) \simeq \bigoplus_\lambda j_{\lambda!*}(M_\lambda)
$$
in $\CHM(S)$ where $j_\lambda:V_\lambda \rightarrow S$ is a locally closed immersion, $V_\lambda$ is regular,
 $M_\lambda$ is a weight $0$ motive over $V_\lambda$ whose $\ell$-adic realisation is a local system up to a shift,
 and $\rho_\ell j_{\lambda!*}(M_\lambda)=j_{\lambda!*}\big(\rho_\ell(M_\lambda)\big)$. 
\end{df}
The interest of this definition is that it does not require a base field,
 and moreover, it makes sense without having to choose
 a prime invertible on $S$; in particular over $\ZZ$. Note also that our formulation implies
 the existence of a decomposition of the form
 \eqref{eq:pDeligne} as in footnote \ref{fn:pDeligne},
 for all primes $\ell$,
 and contains an "independence of $\ell$" result:
 under the existence of a BCH-decomposition,
 the rank of $^pR^nf_*(\QQ_\ell)$ will be independent of $\ell$
 (see also Cor. \ref{cor:independance_ell}).

\begin{num} \textit{Link with Bondarko's version of the motivic conjectures}.
As said before, the definition is partly motvivated by analysing the work of Bondarko \cite{Bon15}.
 In particular, we can conjecturally justify the existence of BCH-decompositions.
 Indeed, the existence of such decompositions would be a consequence
 of the existence of the motivic $t$-structure together with the
 assumption that it is \emph{transversal} (see \cite[Def. 1.4.1]{Bon15}) to the Chow weight structure.\footnote{Note
 also that Bondarko shows that this would be a consequence
 of Beilinson's conjectures (see \emph{op. cit.}, Prop. 4.1.1).
 He also gives an argument to get it as a consequence of
 the standard conjecture D together with Murre's conjectures.
 Note that to get the existence of BCH-decompositions in all
 cases, we need to apply these conjectures  for all residue
 fields of $S$, and all $\ell$-adic realisations! Of course,
 this kind of exercise in juggling between motivic conjectures
 has its limit; we only hope to convince the reader that
 our definition of BCH-decompositions is reasonable.}
 More precisely, to get the BCH-decomposition under the previous
 assumption, one can apply \emph{op. cit.}: 
\begin{itemize}
\item Prop. 1.4.2, point (3) to the weight $0$ motive $h_S(X)$
 to get the analogue of the relative Chow K\"unneth decomposition:
 $h_S(X) \simeq \oplus_i H^i(h_S(X))[-i]$, $H^i$ computed
 for the motivic $t$-structure.
\item Prop. 4.2.3, to get the decomposition for each
 $H^i(h_S(X))$ given it is pure of weight $i$.
\item Rem. 4.2.4, point 2 to get the statement about
 the $\ell$-adic realisation of the BCH-decomposition.
\end{itemize}
Note moreover that Bondarko obtains in his Proposition 4.2.3
 the uniqueness of the set of isomorphism classes of the
 factors of the BCH-decomposition. We have not stated
 this uniqueness property in the above definition. Indeed, to give a 
 proper statement we either need the existence of the heart
 of the motivic t-structure, or conservativity of the $\ell$-adic realisations.\footnote{Note however that in our Theorems A and B below, the involved motives belong to subcategories on which the realisations can be proven to be conservative (Thm. \ref{real}). So, in our cases we \emph{do} get uniqueness of isomorphism classes in the BCH-decompositions.} However,
 we think that establishing BCH-decompositions is a good way to approximate the conjectural motivic t-structure,
 as shown by Deligne's formula \eqref{eq:pDeligne}.
\end{num}

\begin{ex}\label{ex:cellular}
There is at least one easy example where a BCH-decomposition
 exists: this is the case of relative lci cellular morphisms,
 first addressed by Karpenko
 (see \cite[Def. 5.3.1, Cor. 5.3.7]{ADN} for the more general case).\footnote{This is not surprising as one can
 interpret the BCH-decomposition as an attempt to
 find an algebraic analogue of cellular decompositions of differential varieties.}
 Theorem \ref{thm:decomposition1} gives many more examples in the smooth case (see also \ref{ex:BCH-decompositions}),
 and has indeed served as a motivation for the previous definition.
\end{ex}

\subsection{Smooth Artin motives}

It is possible to extend the known results on Artin motives
 over a field to the relative case. Let us start with the definition.
\begin{df}
We define the category of (constructible) smooth Artin motives
 over $S$ as the thick triangulated subcategory of $\DM_c(S,\QQ)$ generated by motives of the form $h_S(X)$ (resp. $h_S(X)(n)$)
 for $X/S$ finite and \'etale.
 We denote this category by $\DMA(S,\QQ)$.
\end{df}
Note that for $X/S$ finite \'etale, $h_S(X)$ coincides with the Voevodsky
 motive $M_S(X)$.

\begin{ex}
Let $k$ be a perfect field.
 Then it follows from \cite[Chap. 5, Rem. 2. after 3.4.1]{FSV}
 (see also \cite{Orgo} for details)
 that there exists a canonical equivalence of triangulated monoidal categories:
$$
\DMA(k,\QQ) \simeq \Der^b\big(\Rep_\QQ(G_k)\big)
$$
where $G_k$ is the absolute Galois group of $k$,
 and $\Rep_\QQ(G_k)$ denotes the continuous representations of $G_k$
 with rational coefficients.
\end{ex}

\begin{num}\label{num:Artin_motives}
Let us recall that one model of $\DM(S,\QQ)$ is obtained as 
 the $\PP^1$-stable $\AA^1$-derived category of the category
 $\Sh(\sm_{S,\et},\QQ)$
 of rational \'etale sheaves on $\sm_S$ (see \cite[16.2.18]{CD3}).

The inclusion $\rho:S_\et \rightarrow \sm_{S,\et}$ of \'etale sites
 induces a fully faithful and exact functor:
$$
\rho_!:\Sh(S_\et,\QQ) \rightarrow \Sh(\sm_{S,\et},\QQ).
$$
Note $\rho_!$ is moreover (symmetric) monoidal.
One deduces a canonical composite functor:
$$
\Der(S_\et,\QQ) \xrightarrow{\rho_!} \Der\big(\Sh(\sm_{S,\et},\QQ)\big)
 \rightarrow \DM(S,\QQ)
$$
the last functor being obtained by projection to the $\AA^1$-localization
 and then taking infinite suspensions.
 We will still denote the latter composite by $\rho_!$.

Then $\rho_!$ is triangulated and monoidal.
 By definition,
 it sends the sheaf represented by a finite \'etale scheme $X/S$
 on $S_\et$ to the same object on $\sm_{S,\et}$, seen as a motivic complex.
 This is just $M_S(X)=h_S(X)$.
\end{num}
\begin{prop}\label{prop:Artin_motives}
Assume $S$ is a regular connected scheme.
 Let $\pi=\pi_1(S_\et)$ be the \'etale
 fundamental group of $S$ associated with some geometric base point.
 Then the functor $\rho_!$ is fully faithful when restricted to the full subcategory
 $\Der^b\big(\Rep_\QQ(\pi)\big)$ and induces an equivalence of 
 triangulated monoidal categories:
$$
\rho_!:\Der^b\big(\Rep_\QQ(\pi)\big) \rightarrow \DMA(S,\QQ).
$$
Moreover, the $\ell$-adic realisation functor restricted
 to $\DMA(S,\QQ)$ lands into the bounded derived category of Artin
 $\ell$-adic Galois representations and the composite functor:
$$
\Der^b\big(\Rep_\QQ(\pi)\big) \xrightarrow{\rho_!} \DMA(S,\QQ)
 \xrightarrow{\rho_\ell} \Der^b\big(\Rep_{\QQ_\ell}(\pi)\big)
$$
is just the extension of scalar functor associated with $\QQ_\ell/\QQ$.
\end{prop}
\begin{proof}
We prove the first assertion: fully faithful nature of the restriction of $\rho_!$.
 Note that the functor $\rho_!$ admits a right adjoint $\rho^*$.
 The functor $\rho^*$ commutes with direct sums: this follows formally as
 $\Der(S_\et,\QQ)$ is compactly generated\footnote{in fact, it is equivalent
 to $\Der(S_\nis,\QQ)$ as $\pi$ is pro-finite; see \cite[Th. 10.5.10]{CD3}},
 and $\rho_!$ sends the compact generators to compact objects.
We have to prove that for all complexes $K, L$ in $\Der^b\big(\Rep_\QQ(\pi)\big)$,
 and say any $n \in \ZZ$ for the next reduction, the map:
$$
\Hom(K,L[n]) \rightarrow \Hom(\rho_!K,\rho_!L[n])=\Hom(K,\rho^*\rho_!L[n])
$$
is an isomorphism.
 Now we use the fact $\Der^b\big(\Rep_\QQ(\pi)\big)$ is generated
 by shifts of sheaves representable by some finite \'etale cover $X/S$.
 So we are reduced to the case $K=\QQ(X)$, $L=\QQ(Y)$ for $X$ and $Y$ \'etale cover of $S$.
 Explicitly, we have to prove that the following map is an isomorphism:
$$
\Hom(\QQ(X),\QQ(Y)[n]) \rightarrow \Hom(h_S(X),h_S(Y)[n]).
$$
 In the monoidal category $\Sh(S_\et,\QQ)$, and therefore in $\Der(S_\et,\QQ)$,
 the sheaf $\QQ(Y)$ is auto-dual. The same result holds for $h_S(X)$ in $\DM(S,\QQ)$
 (as for example $\rho_!$ is monoidal). Applying the formulas
 $\QQ(X) \otimes \QQ(Y)=\QQ(X \times_S Y)$ and $h_S(X) \otimes h_S(Y)=h_S(X \times_S Y)$,
 we are reduced to the case $Y=S$. In other words we have to prove that the 
 canonical map:
\begin{equation}\label{eq:proof_Artin}
H^n(X_\et,\QQ) \rightarrow \Hom(h_S(X),\QQ[n])=H_M^{n,0}(X,\QQ)
\end{equation}
is an isomorphism. Note that $X$ is regular, as it is \'etale over $S$.
 Thus it is geometrically unibranch and we get from \cite[IX, 2.14.1]{SGA4}:
$$
H^n(X_\et,\QQ)=\QQ^{\pi_0(X)} \text{ for } n=0, \ 0 \text{ otherwise.}
$$
Also, according to \cite[14.2.14]{CD3}, we get:
$$
H_M^{n,0}(X,\QQ)=Gr_\gamma^0 K_{-n}(X)_\QQ
=\QQ^{\pi_0(X)} \text{ for } n=0, \ 0 \text{ otherwise,}
$$
and so the map \eqref{eq:proof_Artin} is necessarily an isomorphism.

The other assertions are clear, by definition of the category of smooth Artin motives.
\end{proof}

In the complex case, we get a simpler formulation.
\begin{prop}\label{prop:Artin_motives_complex}
Let $E$ be a field of characteristic zero, $S$ a smooth connected $E$-scheme and $\pi=\pi_1(S^\an)$
 for any choice of base point of $S^\an$.

Then the Betti realisation functor:
$$
\rho_B:\DM_c(S,\QQ) \rightarrow \Der(S^\an,\QQ)
$$
is fully faithful when restricted to $\DMA(S,\QQ)$ and induces an equivalence
 of triangulated monoidal categories:
$$
\rho_B:\DMA(S,\QQ) \rightarrow \Der^b\big(\Rep_\QQ(\pi)\big).
$$
\end{prop}
The proof uses the same argument as in the case of the previous proposition,
 given that $h_S(X)$ is realized to the complex $\derR f_*^\an(\QQ_X)$
 which is concentrated in degree $0$ and equal to the continuous representation
 of $\pi$ represented by the Galois cover $X^\an$ over $S^\an$.

\begin{rem}
The two previous propositions are obviously compatible: in the assumptions of the second one,
 we get according to a theorem of Grothendieck:
$$
\pi_1(S_\et) \simeq \widehat{\pi_1(S^\an)}
$$
where the right hand-side denotes profinite completion.
 In particular, we get equivalence of (abelian semi-simple monoidal) categories:
$$
\Rep_\QQ\big(\pi_1(S_\et)\big) \simeq \Rep_\QQ\left(\widehat{\pi_1(S^\an)}\right)
 \simeq \Rep_\QQ\big(\pi_1(S^\an)\big).
$$
\end{rem}

\subsection{Weights on smooth Artin-Tate motives}

The purpose of this section is to extend results of Wildeshaus,
 \cite{WilAT}
 on Artin-Tate motives over a field to the case of
 a regular base scheme $S$. Moreover,
 for the purpose of our main theorem,
 we will need to restrict the type of allowed Artin motives,
 as in \emph{loc. cit.} So we introduce the next definition.
\begin{df} \label{AT}
Let $S$ be a regular connected scheme
 with \'etale fundamental group $\pi=\pi_1(S_\et)$.
  Let $\cArt$ be a full $\QQ$-linear sub-category of $\Rep_\QQ(\pi)$
 which is stable under retracts.\footnote{And therefore stable under kernel and cokernel as $\Rep_\QQ(\pi)$
 is abelian semi-simple.
 In particular, $\cArt$ is abelian semi-simple.}

We define the triangulated category of smooth Artin-Tate motives
 of type $\cArt$ over $S$
 as the thick triangulated subcategory of $\DM_c(S,\QQ)$
 generated by motives of the form $\rho_!(A)(n)$
 for an object $A$ of $\cArt$ and an integer $n \in \ZZ$
 (see Proposition \ref{prop:Artin_motives} for the definition of $\rho_!$).
 We denote it by $\DMATg(S,\QQ)$.

When $\cArt=\Rep_\QQ(\pi)$,
 the above category is simply the category of smooth Artin-Tate motives
 over $S$, denoted by $\DMAT(S,\QQ)$.

\end{df}

\begin{rem} \label{tens}
Note that compared to Definition 1.6 of \cite{WilAT},
 we do not assume that $\cArt$ is closed under tensor product.
 As a consequence $\DMATg(S,\QQ)$ is not monoidal in general.
\end{rem}

%
\begin{thm}\label{thm:Artin-Tate}
Consider the above notations and assumptions of the above
 definition.

Then the weight structure on $\DM_c(S,\QQ)$ restricts to a weight
 structure on the triangulated sub-category $\DMATg(S,\QQ)$.
Moreover, any motive $M$ in $\DMATg(S,\QQ)$ of weight $0$
 admits a decomposition:
\begin{equation}\label{eq:decomp_AT}
M \simeq \bigoplus_{i \in I} \rho_!(V_i)(n_i)[2n_i]
\end{equation}
where $I$ is a finite set, $V_i$ is a simple Artin representation
 of $\pi$ in $\cArt$ and $n_i$ is an integer.
 (See Paragraph \ref{num:Artin_motives} for $\rho_!$).
\end{thm}

\begin{rem}\label{rem:thm:Artin-Tate}
\begin{enumerate}
\item Another way of stating the second assertion is that
 the canonical functor
\begin{align*}
\Der^b(\cArt) & \rightarrow
  \DMATg(S,\QQ)_{w=0}=\DMATg(S,\QQ) \cap \Chow(S,\QQ) \\
 (V_n)_{n \in \ZZ} & \mapsto \bigoplus_{n \in \ZZ} \rho_!(V_n)(n)[2n]
\end{align*}
is essentially surjective. Contrary to what happens in \cite{WilAT},
 over a perfect field, this functor is not an equivalence of 
 triangulated categories. The problem comes from the
 non-triviality of $CH^n(V)_\QQ$ for $n>0$ and $V/S$ finite \'etale.
 So in particular,
 the preceding functor is an equivalence when $S$ is (regular) semi-local.
\item The decomposition \eqref{eq:decomp_AT} is a particular
 case of the Corti-Hanamura decomposition: 
 If we apply the $\ell$-adic realization functor $\rho_\ell$
 to an Artin representation $V_n$  we get (according to the last assertion of Prop.
 \ref{prop:Artin_motives})
$$
\rho_\ell(M)
 \simeq \bigoplus_{i \in I} (V_i\otimes_\QQ \QQ_\ell)(n_i)[2n_i].
$$
In particular, if the decomposition \eqref{eq:decomp_AT} of $M$
 is not unique. But the pairs $(V_i,n_i)$ for $i \in I$
 are uniquely determined by $M$ (or its realization).
\end{enumerate}
\end{rem}
\begin{proof}
Let us show the first assertion.
 Let $\mathcal H$ be the full $\QQ$-linear sub-category 
 of $\DMATg(S,\QQ)$ whose objects are of the form \eqref{eq:decomp_AT},
 where $V_i$ is a simple object of $\cArt$.
 As $S$ is regular, all motives in $\mathcal H$ are of weight $0$.
 In particular, given such motives $M$, $N$, one has
 $\Hom(M,N[i])=0$ for $i>0$.
 This condition implies that there is a unique weight structure on
 $\DMATg(S,\QQ)$ (see eg. \cite[1.5]{WilChow}),
 whose heart is the pseudo-abelianization $\mathcal K$ of $\mathcal H$.
 It follows from the axioms of weight structures that this weight structure
 is just the restriction of the Chow weight structure,
 thus proving the first assertion (see \cite[Rem. 4.4]{Wil17}).

To prove the second assertion, it is sufficient to prove that $\mathcal K=\mathcal H$,
 \emph{i.e.} that $\mathcal H$ is pseudo-abelian.
 Let $e:M \rightarrow M$ be an idempotent of $\mathcal H$.
 By assumption, we can write $M=\bigoplus_{i=1}^r \rho_!(V_i)(n_i)[2n_i]$
 where $V_i$ is an object of $\cArt$ (semi-simple, but not necessary simple)
 and $n_1<n_2<...<n_r$.
 We prove by induction on the number of factors $r$ of $M$ that $e$ admits a kernel
 in $\mathcal H$.

We treat the case $r=1$.
 According to Proposition \ref{prop:Artin_motives}, and the fact that twists are invertible,
 the category of motives of the form $\rho_!(V_1)(n_1)[2n_1]$ for an object $V_1$ of $\cArt$
 and an integer $n_1$ is equivalent to the abelian (semi-simple) category $\cArt$.
 This implies that any idempotent of a motive of this form admits a kernel,
 giving the case $r=1$.

Assume the result is known for any integer less than $r>1$,
 and prove the case where $M$ has exactly $r$ factors as above.
 Put $P=\rho_!(V_1)(n_1)[2n_1]$ and $Q=\bigoplus_{i=2}^r \rho_!(V_i)(n_i)[2n_i]$
 so that $M=P \oplus Q$.
 Given that decomposition, we can write the idempotent $e$ as a 2 by 2 matrix:
$$
e=\begin{pmatrix}
 a & b \\ c & d
\end{pmatrix}.
$$
Then the map $b:Q \rightarrow P$ belongs to
$$
\oplus_{i=2}^r \Hom\big(\rho_!(V_i)(n_i)[2n_i],\rho(V_1)(n_1)[2n_1])=\CH^{n_1-n_i}(V_i \times_S V_1)
$$
using the auto-duality of $\rho_!(V)$ (see the proof of Prop. \ref{prop:Artin_motives})
 --- in the right hand-side, we identify the sheaf $V_i$ with the finite \'etale $S$-scheme
 which represents it.
 As by assumption, $n_1<n_i$ for $i>1$, we get that $b=0$.
 From the relation $e^2=e$, one deduces that $a$ and $d$ are idempotents, of $P$ and $Q$
 respectively, and one also gets the following constraint on $d$:
$$
ca+dc=c.
$$
By induction, on deduce that $P$ (resp. $Q$) splits into $\Ker(a) \oplus \Ker(1-a)$
 (resp. $\Ker(d) \oplus \Ker(1-d)$) with respect to the idempotent $a$ (resp. $d$)
 and that the map
 $\Ker(a) \oplus \Ker(d) \rightarrow P \oplus Q$
 given by the matrix
 $\begin{pmatrix} 1 & 0 \\ -c & 1\end{pmatrix}$
 is a kernel of $e$ in $\mathcal H$. This concludes the induction step, and the proof.
\end{proof}

\subsection{First decomposition theorem}

For the needs of the proof of our first decomposition
 theorem, we will extend some definitions of \cite{WilMIC}. Let us first recall
 the following theorem, Th. 1.7 of \emph{loc. cit.}
\begin{thm}[Wildeshaus]
Let $j:U \rightarrow S$ be an open immersion.
 Then the weight-exact functor $j^*$ induces an additive exact functor 
$$
j^*:\Chow(S) \rightarrow \Chow(U)
$$
which is essentially surjective and full.
\end{thm}

\begin{num}
Consider an open immersion $j:U \rightarrow S$ and a Chow motive $M$ over $U$.
 A \emph{Chow extension} of $M$ along $j$ will be a pair $(\bar M,\alpha)$
 where $\bar M$ is a Chow motive over $S$ and $\alpha$ is an isomorphism
 $j^*(\bar M)\xrightarrow{\sim} M$. Morphisms of such extensions are defined in
 the obvious way.

According to the previous theorem, extensions of $M$ along $j$ always exist.
 The goal of the work of Wildeshaus is to find (and to define) the
 \emph{intermediate extension} of $M$ along $j$: see \cite{Wil17}, Summary 2.12.
 For the needs of our first decomposition theorem,
 we will use a special type of Chow extensions that we now introduce.
\end{num}
\begin{df}\label{df:fair_ext}
Consider the above notations, assuming that $j$ is dense.
 Then a Chow extension $(\bar M,\alpha)$ of $M$ along $j$ will be called \emph{fair}
 if the induced map
$$
\End(\bar M) \xrightarrow{} \End\big(j^*(\bar M)\big) \xrightarrow{\alpha_*} \End(M)
$$
is an isomorphism --- \emph{i.e.} a monomorphism according to the previous theorem.
\end{df}

The proof of the following result is identical to that of \cite[Th. 3.1(a)]{WilMIC}.
\begin{thm}\label{thm:fair_ext}
Consider the notations of the previous definition.
 Assume a fair Chow extension $(\bar M,\alpha)$ of $M$ along $j$ exists.

Then for any extension $(P,\beta)$ of $M$ along $j$, there exists a decomposition of $P$ of the form:
$$
\psi:P \xrightarrow{\ \sim\ } \bar M \oplus i_*(L_Z)
$$
where $L_Z$ is a Chow motive over $Z$, satisfying the relation:
 $j^*(\psi)=\alpha^{-1} \circ \beta$.

If moreover $(P,\beta)$ is fair, then $L_Z=0$ and the isomorphism $\psi$ is uniquely determined
 by the preceding relation.
\end{thm}

\begin{ex}\label{ex:fair_ext}
A fair Chow extension as above is a particular instance of Wildeshaus's theory
 of intermediate extension $j_{!*}(M)$, as shown for example by the characterizing property
 (4a) of Summary 2.12 of \cite{Wil17}. The preceding theorem can also be interpreted as a
 minimality property of the extension $(\bar M,\alpha)$.

 One cannot always expect that such minimal extensions exist (the correct hope is formulated in \emph{loc. cit.}, Conjecture 3.4; cfr. Rmk. \ref{conj_semiprim}). However, here are some interesting examples.
\begin{enumerate}
\item Assume that $U$ is regular, and the normalization of $X$ is regular.
 Then $\un_U$ admits a fair Chow extension: this is \cite[Th. 3.11(a)]{WilMIC}.
\item Assume that both $U$ and $S$ are regular.
 Let $V/U$ be an \'etale cover. Using a classical terminology,
 we will say that $V$ is \emph{non-ramified along $(S-U)$} if there exists an \'etale cover
 $\bar V/S$ whose restriction to $U$ is $V$.
 In that case, for any integer $n \in \ZZ$,
 the Artin-Tate motive $\bar M=h_S(\bar V)(n)[2n]$ is a fair Chow extension of $M=h_U(V)(n)[2n]$
 along $j$.

Indeed, $\bar M$ is obviously a Chow extension of $M$ along $j$ (using the six functors formalism).
 Moreover, as in the proof of Prop. \ref{prop:Artin_motives}, one gets:
$$
\End(\bar M) \simeq \QQ^{\pi_0(\bar V \times_S \bar V)}, \ \End(M) \simeq \QQ^{\pi_0(V \times_U V)}.
$$
As $V \times_U V$ is a dense open of $\bar V \times_S \bar V$, we get that $\bar M$ is fair
 as claimed.
\end{enumerate}
\end{ex}

\begin{rem}
The reader can check that one can replace, in \cite[Th. 3.11]{WilMIC} points (a) and (b),
 the constant motive $\un_U$ by a smooth Artin-Tate motive $h_U(V)$
 such that $V$ is unramified along $(S-U)$.
\end{rem}

\begin{num}
Recall also that a Chow motive over a field $k$ is said to be \emph{Tate}
 if it is isomorphic to a finite sum of motives of the form $\un(i)[2i]$
 for an integer $i \in \ZZ$. By extension, a smooth proper $k$-scheme
 is said to be Tate if its associated Chow motive is Tate.
\end{num}
\begin{thm}\label{thm:decomposition1}
Let $S$ be a regular connected scheme with \'etale fundamental group $\pi$.
Let $f:X \rightarrow S$ be a smooth proper morphism whose geometric fibers are Tate.

Then there exists an isomorphisms of motives over $S$:
$$
h_S(X) \simeq \bigoplus_{i \in I} \rho_!(V_i)(-n_i)[-2n_i]
$$
where $I$ is a finite set, $V_i$ is a simple Artin representation of $\pi$ and $n_i$ is a non-negative integer.
 In other words, $h_S(X)$ is a smooth Artin-Tate motive over $S$ of weight $0$ --- see Theorem \ref{thm:Artin-Tate}.

Moreover, this decomposition is a BCH-decomposition of the Chow motive $h_S(X)$ (see Def. \ref{df:BCH}), and the set $\big\{\big([V_i],n_i\big), i \in I \big\}$ where $[V_i]$ denotes the isomorphism class in $\Rep_\QQ(\pi)$
 (or what amount to the same in the corresponding category of simple \'etale $S$-covers) is uniquely
 determined by the property that for any prime integer $\ell$ and any $n \geq 0$:
\begin{equation}\label{eq:thm:decomposition_ell}
\derR^{2n} f'_*(\QQ_\ell) \simeq \bigoplus_{i \in I \mid n_i=n} V'_i \otimes_\QQ \QQ_\ell(-n_i)
\end{equation}
where $f'$ and $V'_i$ are pullback of $f$ and $V_i$ along the open immersion
 $U[\ell^{-1}] \rightarrow U$.
\end{thm}
\begin{proof}
Let $\bar \eta$ be a geometric generic point of $S$.
 By assumption, one gets:
$$
h_{\bar \eta}(X_{\bar \eta}) \simeq \sum_{i \in I} \un_{\bar \eta}(-n_i)[-2n_i].
$$
According to the continuity property of $\DM_c$ (see Proposition \ref{prop:cont}),
 there exists a dense open $j:U \rightarrow S$ and an \'etale cover $p:V \rightarrow U$ such that
 the above isomorphism lifts to:
$$
h_V(X_V) \simeq \bigoplus_{j \in J} \un_V(-m_j)[-2m_j].
$$
In particular, $h_U(X_V)=p_*(h_V(X_V)) \simeq \sum_{j \in J} p_*(\un_V)(-m_j)[-2m_j]$
 is a smooth Artin-Tate motive over $U$.
 As $p$ is finite \'etale, the natural map
$$
p^*:h_U(X_U) \rightarrow h_U(X_V)
$$
is a split monomorphism, with splitting $\frac 1 d.p_*$ where $p_*$ is the Gysin morphism
 associated with $p$ (see e.g. \cite[13.7.4 and 13.7.6]{CD3}). In particular,
 $h_U(X_U)$ is a smooth Artin-Tate motive, and it follows from Theorem \ref{thm:Artin-Tate}
 that there exists a decomposition:
$$
h_U(X_U) \simeq \bigoplus_{i \in I} \rho_!(W_i)(-n_i)[-2n_i]
$$
where $W_i$ is a simple Artin representation of $\pi_1(U)$.
 Given a prime $\ell$, we let  $U'=U[\ell^{-1}]$, $f_{U'}$ the pullback of $f$ over $U'$,
 $W'_i$ the Artin representation of $\pi_1(U')$ induced by $W_i$.
 According to Remark \ref{rem:thm:Artin-Tate}(2), and because
 $\rho_\ell(h_{U'}(X_{U'})) \simeq \derR f_{U'*}(\QQ_\ell)$, one gets:
\begin{equation}\label{eq:decomposition1a}
\derR^{2n} f_{U'*}(\QQ_\ell) \simeq \bigoplus_{i \in I \mid n_i=n} W'_i \otimes_\QQ \QQ_\ell(-n_i).
\end{equation}
This is the decomposition of the locally constant sheaf $\derR^{2n} f_{U'*}(\QQ_\ell)$ into
 semi-simple components (beware the $W'_i$ might not be simple).
 As this sheaf admits an extension to all $S$, namely $\derR^{2n} f'_{*}(\QQ_\ell)$, $f'=f[\ell^{-1}]$,
 it follows that each representation $W'_i$ is unramified along $(S-U)$.
 As this is true for any prime $\ell$, one deduces that $W_i$ is unramified along $(S-U)$;
 \emph{i.e.} it admits an extension $V_i$ to $S$.
 According to Example \ref{ex:fair_ext}(2), one deduces that
$$
\bigoplus_{i \in I} \rho_!(V_i)(-n_i)[-2n_i]
$$
is a fair Chow extension of $h_U(X_U)$. As $h_S(X)$ is obviously a Chow extension of $h_U(X_U)$, one deduces
 from Theorem \ref{thm:fair_ext} that there exists a decomposition:
\begin{equation}\label{eq:decomposition1b}
h_S(X) \simeq \bigoplus_{i \in I} \rho_!(V_i)(-n_i)[-2n_i] \oplus i_*(M_Z)
\end{equation}
for some Chow motive $M_Z$ over $Z$. 
 Note that Relation \eqref{eq:decomposition1a} implies Relation \eqref{eq:thm:decomposition_ell}.\footnote{In particular,
 we know that $\rho_\ell(M_Z)$ vanishes but this is not sufficient to conclude
 (as we do not know yet that the $\ell$-realization is conservative).} 

It remains to prove that $M_Z=0$.
 As the morphism
 $p:S'=\sum_{\ell\text{ prime}} S[\ell^{-1}] \rightarrow S$
 is a pro-open cover, the pullback functor
 $p^*:\DM(S) \rightarrow \DM(S')$ is conservative
 (use the continuity property of $\DM$ \cite[Th. 14.3.1]{CD3} 
 and the Zariski separation property
 as in the proof of \cite[Prop. 4.3.9]{CD3}).
 In particular, we can fix a prime $\ell$ and work over $S[\ell^{-1}]$.
 To simplify notation,
 let us assume $S=S[\ell^{-1}]$.
 Consider a point $x \in Z$, and let $i_x:\{x\} \rightarrow Z$ be the canonical immersion.
 Let $\bar x$ be a geometric point over $x$. By assumption, $h_{\bar x}(X_{\bar x})$ is a Tate motive.
 In other words, the motive $i_x^*h_S(X)=h_x(X_x)$ is an Artin-Tate motive over the residue field $\kappa(x)$.
 Decomposition \eqref{eq:decomposition1b} gives:
$$
h_x(X_x) \simeq \bigoplus_{i \in I} \rho_!(V_{i,x})(-n_i)[-2n_i] \oplus i_x^*(M_Z)
$$
where $V_{i,x}$ denotes the pullback of $V_i$ to $x$ (seen as an \'etale sheaf).
 As recalled in Remark \ref{rem:thm:Artin-Tate}(2), applied over $\kappa(x)$,
 the $\ell$-adic realization of $h_x(X_x)$ determines the factors of the decomposition of $h_x(X_x)$
 into twists of Artin motives (up to isomorphisms and permutations).
 Thus relation \eqref{eq:thm:decomposition_ell} (which we have already established), specialized at $x$ using the smooth base change theorem
 in $\ell$-adic \'etale cohomology, implies that the set of isomorphism classes of the $V_{i,x}$
 describes all the possible factors of the weight $0$ Artin-Tate motive $h_x(X_x)$ and this implies $i_x^*(M_Z)=0$.
One concludes as $(i_x^*)_{x \in Z}$ is a conservative family of functors on $\DM_c(Z,\QQ)$
 (see \cite[4.3.17]{CD3}).  
\end{proof}

\begin{rem}
Note also that we deduce from Example \ref{ex:fair_ext} that for any dense open immersion $j:U \rightarrow S$,
 $h_S(X)$ is a fair Chow extension of $h_U(X_U)$. In particular, $h_S(X)=j_{!*}(h_U(X_U)$.
\end{rem}

\begin{ex}\label{ex:BCH-decompositions}
Let $S$ be a regular connected scheme. 
 The previous theorem applies in the following situations:
\begin{enumerate}
\item Relative Severi-Brauer $S$-schemes: a proper smooth morphism
 $f:X \rightarrow S$
 whose geometric fibers are isomorphic to $\PP^r$;
 see \cite[Section 2]{Bernardara} for more discussion.
\item \'Etale-local cellular $S$-schemes: we have already met relative lci
 cellular schemes (Example \ref{ex:cellular}). The previous theorem
 applies more generally to morphisms $f:X \rightarrow S$ such that
 there exists an \'etale cover $p:X' \rightarrow X$ such that $f \circ p$
 is a relative lci cellular scheme in the sense of \emph{loc. cit.}
\end{enumerate}
Point (2) contains in particular the cases of smooth families of projective quadrics,
 and smooth families of projective homogeneous varieties.
\end{ex}

\begin{cor}\label{cor:independance_ell}
Under the assumption of the previous theorem,
 for any integer $n \geq 0$,  the integer:
$$
\dim_{\QQ_\ell}\big(\derR^n f_*(\QQ_\ell)\big)
$$
is independent of the prime $\ell$ invertible on $S$.
\end{cor}

\begin{num}\label{num:clifford_aglebra&Prym}
An important corollary for us is the case of smooth quadrics:
 following \cite[XII, Def. 2.4]{SGA7}, a smooth quadric
 over a scheme $S$ is a smooth proper morphism $f:X \rightarrow S$
 whose geometric fibers are smooth quadric hypersurfaces in the
 classical sense.

Let us introduce some notations in order to state
 the next corollary.
Assume $f$ has constant relative dimension $n$.
 According to \emph{op. cit.}, 2.6, there exists a Severi Brauer 
 $S$-scheme $P(X)$ which contains $X$ as an effective Cartier
 divisor of degree $2$. If $n=2m$, the center of the Clifford
 $\mathcal O_S$-algebra associated with the closed pair $\big(P(X),X\big)$
 is an \'etale cover $Z(X)$ over $S$ of degree $2$; \emph{op. cit.}, 2.7.

Then, applying Theorem \ref{thm:decomposition1} and \cite[Th. 3.3]{SGA7}
 one gets:
\end{num}
\begin{cor}\label{cor:sm_quadrics}
Consider the above notations. Assume that $S$ is regular
 and that there exists a prime $\ell$ invertible on $S$.

Then there exists a decomposition of Chow motives over $S$ as follows:
$$
h_S(X) \simeq
\begin{cases}
\bigoplus_{i=0}^n \un_S(i)[2i] & \text{if } n=2m+1, \\
h_S\big(Z(X)\big) \oplus \bigoplus_{i=0, i \neq m}^n \un_S(i)[2i] & \text{if } n=2m.
\end{cases}
$$
\end{cor}

\begin{ex} \label{arith_smprquad}
Let us give some concrete examples illustrating the above Corollary when $S$ is an arithmetic scheme.
\begin{enumerate}[wide, labelindent=0pt, label=(\arabic*)]
\item It is known that smooth quadrics over $\Spec \BZ$ exist: one way to see this has been explained by Buzzard in \cite{9605}. One considers a quadric hypersurface $X$ in $\BP^N_{\BQ}$ whose Hessian matrix $H$ is a symmetric, integer valued $(N+1) \times (N+1)$ matrix with $\det H= \pm 1$ and even entries down the diagonal; e.g., the Gram matrix of any even unimodular lattice (these exist in any positive dimension divisible by 8) provides such an $H$. As $\det H$ is non-zero mod $p$ for any prime $p$, the quadric $X$ admits a proper and smooth model over $\BZ$. 
\item Since $\Spec \BZ$ is simply connected (a consequence of Minkowski's theorem), the previous construction yields quadrics whose motive decomposes, according to Cor. \ref{cor:sm_quadrics}, as a sum of Tate motives (the Artin factors are necessarily trivial). In order to find cases where genuine Artin-Tate motives appear, we have to place ourselves over a non-simply connected base. Let then $F$ be a number field with ring of integers $\mathcal{O}_F$, such that $S:= \Spec{\mathcal{O}_F}$ has non-trivial \'etale fundamental group. In this setting, examples abound, as we are going to argue following a line of thought suggested to us by G. Ancona. 

Let $N$ be a positive integer. Analogously to Example (1) above, we look to a $\mathcal{O}_F$-valued $2N \times 2N$ matrix $H$, such that any entry on the diagonal is divisible by 2, and $\det H$ is invertible in $\mathcal{O}_F$. We may take for example $N=2$, $\mathcal{O}_F=\BZ[\sqrt -5]$ (we have then $\pi_{1,\et}(S) \simeq \BZ / 2 \BZ$) and 
\begin{equation}\label{exmat}
H=\left(
\begin{array}{cc}
A & 0 \\
0 & B
\end{array}
\right), 
A=\left(
\begin{array}{cc}
-4\sqrt -5 & 9 \\
9 & 4\sqrt -5 
\end{array}
\right), 
B=\left(
\begin{array}{cc}
2 & \sqrt -5 \\
\sqrt -5 & -2
\end{array}
\right) 
\end{equation}

The matrix $H$ is then the Hessian matrix of a smooth quadric $X$ in $\BP^{2N-1}_{S}$. We want to exhibit examples where the motive $h_S(Z(X))$ of Cor. \ref{cor:sm_quadrics} is non-Tate. By proper base change and Thm. \ref{prop:Artin_motives}, it is enough to make sure that there exists a prime $\mathfrak{p}$ of $\mathcal{O}_F$ such that the Galois representation on $H_{\et}^{N-1}(X_{\mathfrak{p}}, \BQ_{\ell})$, for $\ell \neq p:= \car \mathcal{O}_F/\mathfrak{p}$, is non-trivial. The Lefschetz fixed point formula shows that the latter property is equivalent to the condition
$$
\mathrm{Card}\big(X_{\mathfrak{p}}(\BF_q)\big) \neq 1 + q +...+2q^{N-1}+...q^{2N-2},
$$
where $\BF_q$ denotes the field $\mathcal{O}_F/\mathfrak{p}$ and $q=p^r$ for some $r$. This can be checked in specific cases,
 using the following elementary fact: a quadratic form in at least three variables over a finite field always admits an isotropic vector,
 hence always contains an hyperbolic plane.

We make explicit this principle when $N=2$. Let $m$ be the quadratic form in four variables over $\BF_q$ obtained by reducing the coefficients of $H$ modulo $\mathfrak{p}$.
 By the fact just recalled, $m$ can be written in the form $ut=rx^2+sy^2$, with $r,s \in \BF^{\times}_q$.
 Now a computation shows that if $rx^2+sy^2$ has no isotropic vectors, then $\mathrm{Card}\big(X_{\mathfrak{p}}(\BF_q)\big)=1+q^2$,
 thus giving us the desired quadric. For example, if we choose $A$ and $B$ as in \eqref{exmat}
 and suppose that $p \neq 2$, we have $r=s=1$,
 and the fact that the bilinear form $x^2+y^2$ has no isotropic vectors is equivalent to ask that $1=\disc(x^2+y^2) \neq -1 \mod (\BF^{\times}_q)^2$,
 which is equivalent to ask $-1$ not being a square in $\BF_q$.
 Now choose as $\mathfrak p$ one of the two prime ideals above the completely split prime $p=3$. Then, $\BF_q = \BZ/3\BZ$, and $-1$ is not a square mod 3. 

\end{enumerate}
\end{ex}

\section{Recall and complement on Wildeshaus's motivic intermediate extensions}

\subsection{Semi-primary categories and motivic intermediate extensions.} 
Recall the following definition. 

\begin{df} \label{semipr}
A $\BQ$-linear category $\FC$ is \emph{semi-primary} if 
\begin{enumerate}[wide, labelindent=0pt, label=(\arabic*)]
\item for all objects $B$ of $\FC$, the \emph{radical}
\[
\rad_{\FC}(B,B) := \{ f \in \Hom_{\FC}(B,B) \vert \forall \ g \in \Hom_{\FC}(B,B), \id_B-gf \ \mbox{is invertible} \} 
\]
is nilpotent;
\item the quotient category $\FC / \rad_{\FC}$ is semisimple. 
\end{enumerate}
\end{df}

Adopt the notations of subsection \ref{sec:weight}. We are now going to explain how the notion of semi-primality leads to a definition of an intermediate extension functor, following Wildeshaus. 

Let $\cdot$ denote any of the schemes $U$, $S$ or $Z$, and fix $\mathcal{C}(\cdot)$ full pseudo-abelian subcategories of the categories $\DM_c(\cdot, \BQ)$, related by gluing. Assume that they inherit a weight structure (automatically compatible with the gluing) from the restrictions of the motivic weight structure. The subscript $w=0$ will mean that we are taking the \emph{heart} of such weight structures. 

Moreover, denote by $\mathcal{C}(S)_{w=0}^u$ the quotient of the category $\mathcal{C}(S)_{w=0}$ by the two-sided ideal\footnote{Note that in our setting, this ideal will always be contained in $\rad_{\mathcal{C}(S)_{w=0}}$ (\cite[Cor. 1.5 (a)]{Wil17}).} $\Fg$ generated by $\Hom_{\mathcal{C}(S)}(A, i_*B)$ and $\Hom_{\mathcal{C}(S)}(i_*B, A)$, with $(A,B)$ varying on the collection of objects of $\mathcal{C}(S)_{w=0} \times \mathcal{C}(Z)_{w=0}$ such that $A$ admits no non-zero direct factor belonging to $\mathcal{C}(Z)_{w=0}$. Finally, denote by $\mathcal{C}(Z)_{w=0}^u$ the quotient of the category $\mathcal{C}(Z)_{w=0}$ by the restriction of $\Fg$ to $\mathcal{C}(Z)_{w=0}$ (with respect to the fully faithful inclusion $i_* : \mathcal{C}(Z)_{w=0} \hookrightarrow \mathcal{C}(S)_{w=0}$).

\begin{thm}{\cite[Theorem 2.9]{Wil17}} \label{semiprim}
\begin{enumerate}[wide, labelindent=0pt, label=(\arabic*)]
\item If $\mathcal{C}(Z)_{w=0}$ is semi-primary, then the functors $j^*$ and $i_*$ induce a canonical equivalence of categories
\begin{equation} \label{eqsum}
\mathcal{C}(S)_{w=0}^u \simeq \mathcal{C}(U)_{w=0} \times \mathcal{C}(Z)_{w=0}^u
\end{equation}
\item \label{glue} If both $\mathcal{C}(Z)_{w=0}$ and $\mathcal{C}(U)_{w=0}$ are semi-primary, then so is $\mathcal{C}(S)_{w=0}$.
\end{enumerate}
\end{thm}

\begin{df} \label{intext}
Suppose that $\mathcal{C}(Z)_{w=0}$ is semi-primary. The \emph{intermediate extension} is the fully faithful functor 
\begin{equation*}
j_{!*}: \mathcal{C}(U)_{w=0} \hookrightarrow \mathcal{C}(S)_{w=0}^u
\end{equation*}
corresponding to the functor $(\id_{\mathcal{C}(U)_{w=0}}, 0)$ under the equivalence of categories \eqref{eqsum}. 
\end{df}

It follows from its very definition that the motivic intermediate extension functor $j_{!*}$ enjoys the following property (\cite[Summary 2.12 (b)]{Wil17}): 

\begin{prop} \label{scind}
Consider the notation and the assumptions of the previous definition. Then, any object $M$ of $\mathcal{C}(S)_{w=0}$ is isomorphic to a direct sum $j_{!*}M_U \oplus i_*N$, for an object $M_U$ of $\mathcal{C}(U)_{w=0}$ and an object $N$ of $\mathcal{C}(Z)_{w=0}$. The object $M_U$ is such that $j^*M \simeq M_U$ (hence unique up to unique isomorphism) and $N$ is unique up to an isomorphism, which becomes unique in $\mathcal{C}(Z)^u_{w=0}$. 
\end{prop}

\begin{rem} \label{conj_semiprim}
It is believed that for any $S$, the heart $\Chow(S)$ of the motivic weight structure on $\DM_c(S,\BQ)$ is semi-primary (cfr. \cite[Conj. 3.4]{Wil17}). This conjecture is at the moment completely out of reach, but it would permit, by choosing as $\mathcal{C}(\cdot)$ the whole of the categories $\DM_c(\cdot,\BQ)$, to define (up to non-unique isomorphism) the intermediate extension to $S$ of any Chow motive on $U$. When $S$ is the spectrum of a field, the reader should compare this conjecture with the nilpotency conjecture stated at the beginning of Section \ref{sec:dec}. 
\end{rem}

\subsection{Semi-primary categories of Chow motives.}\label{sec:semi-prim_ChowMot}

The aim of this section is to single out some subcategories of Chow motives which can actually be shown to be semi-primary and which will be suited for our geometric applications. We will adapt Wildeshaus' methods from \cite{Wil17}, in order to show semiprimality of subcategories which are different\footnote{See Rmk. \ref{diff} for a comment on these differences.} from the ones considered in \emph{op. cit.}. 

Fix a scheme $S$ admitting a \emph{good stratification} $\FS$, i.e. such that $S$ may be written as a finite (set-theoretic) disjoint union $\bigsqcup_{\sigma \in \FS} S_{\sigma}$ of locally closed subschemes such that the closure $\overline{S_{\sigma}}$ of each stratum $S_{\sigma}$ is a union of strata. 

We make the following assumption on our good stratification: 
\begin{assumption} \label{assreg}
For all $\sigma \in \FS$, the strata $S_{\phi}$ are nilregular\footnote{I.e., the underlying reduced scheme is regular.}, with nilregular closure.
\end{assumption}

Fix a stratum $S_{\sigma}$ and consider the categories $\DMATg(S_{\sigma}, \BQ)$ introduced in Definition \ref{AT}.  

\begin{df} \label{wr}
The category of \emph{tamely ramified} smooth Artin-Tate motives over $S_{\sigma}$ is the category $\DMAwT(S_{\sigma}, \BQ)$ obtained by choosing $\CA=\mbox{A}_{\mbox{\tiny{tr}}}$, where $\mbox{A}_{\mbox{\tiny{tr}}}$ is the full subcategory of direct factors of objects corresponding to finite \'etale morphisms $q:X_{S_{\sigma}} \rightarrow S_{\sigma}$ which are \emph{tamely ramified in codimension 1} (\cite[58.31, tag 0BSE]{SP})).
\end{df}

\begin{rem}
\label{rem: tr}
In the previous definition, suppose that $S_{\sigma}$ is a scheme of characteristic zero. Then any finite étale morphism with target $S_{\sigma}$ is tamely ramified in codimension 1, and the category $\DMAwT(S_{\sigma}, \BQ)$ equals $\DMAT(S_{\sigma}, \BQ)$.
\end{rem}

\begin{thm}\label{constmot}
Let $S$ be a scheme with a good stratification $\FS$ satisfying Assumption \ref{assreg}. Then:
\begin{enumerate}[wide, labelindent=0pt, label=(\arabic*)]
\item the categories $\DMAwT(S_{\sigma},\BQ)$ of tamely ramified smooth Artin-Tate motives over $S_{\sigma}$, $\sigma \in \FS$, can be glued to give a full, triangulated sub-category $\DMATS$ of $\DM_c(S, \BQ)$, called the category of \emph{$\FS$-constructible tamely ramified Artin-Tate motives} over $S$. This subcategory is stable under formation of direct factors. 
\item Let $M \in \DM_c(S,\BQ)$. Then the following conditions are equivalent.
\begin{enumerate}[wide, labelindent=0pt]
\item $M \in \DMATS$.
\item $j^*M \in \DMAwT(S_{\sigma},\BQ)$ for all $\sigma \in \FS$, where $j$ denotes the immersion $S_{\sigma} \hookrightarrow S$.
\item $j^!M \in \DMAwT(S_{\sigma},\BQ)$ for all $\sigma \in \FS$.
\end{enumerate}
In particular, the triangulated category $\DMAT(S,\BQ)$ of smooth Artin-Tate motives over $S$ is contained in $\DMATS$.
\item The category $\DMATS$ is pseudo-abelian. 
\end{enumerate} 
\end{thm}
\begin{proof}
One proceeds by induction on the number of strata. If there is only one stratum, all claims are trivial, except for the last claim of part (1) and for
part (3). Of course, it suffices to prove part (3). But in the case we are treating, the category in question is just $\DMAT(S, \BQ)$, and it has been shown in the proof of Thm. \ref{thm:Artin-Tate} that this category has a (bounded) weight structure whose heart is pseudo-abelian. Hence (by \cite[Lemma 5.2.1]{Bon10}), the category $\DMAT(S, \BQ)$ itself is pseudo-abelian. 

As for the induction step, we have that the theorem is true for the complement $Z$ of any open stratum $U$, with its stratification $\FS_{Z}$, by the induction hypothesis. Write $j_{U}$, resp. $i_{Z}$, for the open, resp. locally closed immersion of $U$, resp. $Z$, in $S$.

In order to prove (1), we will prove that the criterion given by \cite[Prop. 4.1. (a)]{Wil17} is satisfied; it says that $\DMAwT(U, \BQ)$ and $\DMATSZ$ can be glued if and only if for all objects $M \in \DMAwT(U, \BQ)$, $i_{Z}^* j_{U,*} M$ belongs to $\DMATSZ$. 
For this, we can suppose that the closure of $U$ in $S$ is the whole of $S$, that $S$ is connected, and that $S$ and $Z$ are regular
 (and not just nilregular, cfr. Assumption \ref{assreg}). Then, we take a direct factor $M_{U}$ of some motive $h_{U}(X_{U} / U)$,
 with $q:X_{U} \rightarrow U$ finite \'etale and tamely ramified in codimension 1. Suppose now that $Z$ is of codimension 1 in $S$. Then by \cite[Lem. 58.31.5, tag 0EYH]{SP}, the normalization $q ^\prime :X \rightarrow S$ of $S$ in $X_U$ is a finite morphism such that the restriction of $X$ to $U$ is isomorphic to $X_U$ and such that $Z^\prime:=((q^{\prime})^{-1}(Z))_{\red}$ is an effective Cartier divisor which is a regular scheme. 
Now $X_U$ is regular by assumption. Moreover, any point of $Z^\prime$, being regular on an effective Cartier divisor, is regular as a point of $X$ (\cite[Lem. 10.106.7, tag 00NU]{SP}). Hence, $X$ is regular, and we can apply absolute purity to the
closed immersion $Z^\prime \hookrightarrow X$: by proper base change and the fact that the local form of $q^\prime$ shows that it induces an isomorphism $Z^\prime \simeq Z$, we get that for any $p \in \BZ$, both $i_{Z}^! h_{S}(X /
S)(p)$ and $i_{Z}^* h_{S}(X / S)(p)$ belong to $\DMATSZ$. Now, again using proper base change, we see that $i_{Z}^* j_{U,*} h_{U}(X_{U} /
U)(p)$ is a cone of the canonical morphism $i_{Z}^! h_{S}(X / S)(p) \rightarrow i_{Z}^* h_{S}(X / S)(p)$. Hence, it belongs to $\DMATSZ$.
But the latter category is stable under direct factors by the induction hypothesis, so $i_{Z}^* j_{U,*} M_{U}(p)$ belongs to it as well. This means that the
functor $i_{Z}^* j_{U,*}$ maps the generators of $\DMAwT(U, \BQ)$ to $\DMATSZ$. As a consequence, the whole of $\DMAwT(U, \BQ)$ is mapped to
the latter category under $i_{Z}^* j_{U,*}$, and the criterion is fulfilled. 

If $Z$ is of codimension $\geq 2$ in $S$, then $q$ extends to a finite \'etale morphism $q^\prime:X \rightarrow S$ by purity of the branch locus (see \cite[Lem. 53.20.4, tag 0BMB]{SP}) and we conclude by reasoning as above.

The proof of the remaining points carries over word by word from the proof of the analogous points in \cite[Thm. 4.5]{Wil17}.  
\end{proof}


Let us now set up a slightly different notation, in order to treat the more general geometric situations which we are interested in. 
Suppose that $S$ is equipped with a good stratification, denoted from now on by $\Phi$. The subcategories we are interested in will depend on the choice of a proper morphism $\pi: S^{\prime} \rightarrow S$ from a scheme $S^{\prime}$ which admits a good stratification $\FS$, such that the preimage via $\pi$ of any stratum of $\Phi$ is a union of strata (i.e. $\pi$ is a \emph{morphism of good stratifications}). Whenever $T$ is a subscheme of $S$, we will denote by $\FS_T$ the stratification induced by $\FS$ on $S^\prime_T$.  

It will be necessary to make the following assumptions on the nature of the stratifications $\FS$ and $\Phi$ and of the morphism $\pi$: 
\begin{assumption}{(cfr. \cite[Ass. 5.6]{Wil17})} \label{asspmgs}
\begin{enumerate}[wide, labelindent=0pt, label=(\arabic*)]
\item the good stratification $\FS$ on $S^\prime$ satisfies Assumption \ref{assreg}; \label{upreg}
\item for all $\phi \in \Phi$, the strata $S_{\phi}$ are nilregular; \label{lowreg}
\item \label{shapemor} the morphism $\pi$ is surjective, and for all $\phi \in \Phi$ and $\sigma \in \FS$ such that $S^{\prime}_{\sigma}$ is a stratum of $\pi^{-1}(S_{\phi})$, the morphism $\pi_{\sigma}:S^{\prime}_{\sigma} \rightarrow S_{\phi}$
is proper with geometrically connected fibres, smooth, and such that  
the motive $h(S^{\prime}_{\sigma} / S_{\phi})$  belongs to the category of Tate motives over $S_{\phi}$. 
\end{enumerate}
\end{assumption}

\begin{rem} \label{chowtate}
Since we ask $\pi_{\sigma}$ to be proper and smooth, the previous assumption actually implies that $h(S^{\prime}_{\sigma} / S_{\phi})$ belongs to the category of \emph{weight zero}, \emph{smooth} Tate motives over $S_{\phi}$.
\end{rem}


Fix a morphism $\pi:S^{\prime} \rightarrow  S$ of good stratifications $\Phi$ and $\FS$, satisfying Assumption \ref{asspmgs}.\ref{upreg}. 

\begin{df}\label{imstrmot}
The category $\DMATSP$ is the pseudo-abelian completion of the strict, full, $\BQ$-linear triangulated subcategory of $\DM_c(S,\BQ)$ generated by the images under $\pi_*$ of the objects of $\DMATS$. 
\end{df}

Suppose moreover that Assumption \ref{asspmgs}.\ref{lowreg} is satisfied. Then, reasoning in the same way as in the proof \cite[Cor. 4.11]{Wil17}, we see: 
\begin{lm}\label{ws}
\begin{enumerate}[wide, labelindent=0pt, label=(\arabic*)] 
\item 
The restriction of the motivic weight structure on $\DM_c(S, \BQ)$ induces a bounded weight structure on $\DMATSP$. For $\pi= id$, this gives a bounded weight structure on $\DMATS$.
\item \label{wsheart} The heart of the above weight structure on $\DMATSP$ is the pseudo-Abelian completion of the strict, full, $\BQ$-linear additive subcategory of $\DM_c(S,\BQ)$ generated by the images under $\pi_*$ of the objects of the heart of the weight structure on $\DMATS$. 
\end{enumerate}
\end{lm}

\begin{df} The heart of the weight structure on $\DMATSP$ given by the preceding Lemma is denoted by $\ATSP$.
\end{df}
%
%

Now we can finally state the result that we want to employ. 

\begin{thm} \label{artatesp}
Let $\pi:S^\prime \rightarrow S$ be a proper morphism of good stratifications $\FS$, $\Phi$ satisfying Assumption \ref{asspmgs}. Then, the category $\ATSP$ is semi-primary. 
\end{thm}
\begin{proof}
By Thm. \ref{constmot} and proper base change (cfr. the analogous \cite[Cor. 4.10 (b)]{Wil17}), the category $\DMATSP$ is obtained by gluing the categories $\mbox{DM}^{\mbox{\tiny{AT}}}_{\mbox{\tiny{tr}}}(\FS_{S_{\phi}} / S_{\phi})$ for all $\phi \in \Phi$. Hence, to prove our claim, it is enough to prove semi-primality of $\mbox{AT}_{\mbox{\tiny{tr}}}(\FS_{S_{\phi}} / S_{\phi})$ for each stratum $S_{\phi}$ and then to apply Thm. \ref{semiprim}.\ref{glue} iteratively. 

We first observe that by Lemma \ref{ws}.\ref{wsheart}, the category $\mbox{AT}_{\mbox{\tiny{tr}}}(\FS_{S_{\phi}} / S_{\phi})$ is the pseudo-Abelian completion of the strict, full, $\BQ$-linear triangulated subcategory of $\DM_c(S_{\phi})$ of objects, which are isomorphic to images under $\pi_{\sigma}$ of tamely ramified Artin-Tate motives over $S^\prime_{\sigma}$, for $\sigma \in \FS$ such that $S^\prime_{\sigma}$ is a stratum of $\pi^{-1}(S_{\phi})$. This implies that, by reasoning as in \cite[Thm. 5.4]{Wil17}, the claim will follow as soon as we prove that the objects of the latter form are \emph{finite dimensional} in the sense of Kimura. 

Choose a couple of strata $\pi_{\sigma}: S_{\sigma}^\prime \rightarrow S_{\phi}$ as above and take a tamely ramified finite \'etale morphism $q:D \rightarrow S^\prime_{\sigma}$. Then, consider the Stein factorization of the morphism $\pi_{\sigma} \circ q$ (in the form provided by, for example, \cite[Thm. 6.2 (2)]{Cad13}). We obtain a commutative diagram 
\begin{center}
\begin{tabular}{c}
\xymatrix{
D \ar[r]^{q} \ar[d]_{p} & S^\prime_{\sigma} \ar[d]^{\pi_{\sigma}} \\
\tilde{S}_{\phi} \ar[r]^{r} & S_{\phi}
}
\end{tabular}
\end{center}
where $p$ is proper with connected fibres and $r$ is finite \'etale. Moreover, the fiber of $r$ over each point $s$ of $S_{\phi}$ is in set-theoretic bijection with the set of connected components of the fiber of $\pi_{\sigma} \circ q$ over $s$. Thus, because of our assumption \ref{asspmgs}.\ref{shapemor} on the properties of $\pi_{\sigma}$, the degree of $r$ is the same as the degree of $q$, say equal to $d$. We get a diagram 
\begin{center}
\begin{tabular}{c}
\xymatrix{
D \ar[ddr]_{p} \ar[dr]^{\iota} \ar[drr]^{q} & &  \\
& \tilde{S}_{\phi} \times_{S_{\phi}} S^\prime_{\sigma} \ar[r]_{q^\prime} \ar[d]^{p^\prime} & S^\prime_{\sigma} \ar[d]^{\pi_{\sigma}} \\
& \tilde{S}_{\phi} \ar[r]^{r} & S_{\phi}
}
\end{tabular}
\end{center}
where $q^\prime$ is finite \'etale of degree $d$, so that $\iota$ has to be finite \'etale of degree 1, i.e. it embeds $D$ as a connected component of $\tilde{S}_{\phi} \times_{S_{\phi}} S^\prime_{\sigma}$. Call the latter scheme $D^\prime$. Then, using proper base change and the fact that $q^\prime$ and $r$ are finite \'etale, we get
\[
h(D^\prime / \tilde{S}_{\phi})=p^\prime_* \one_{D^\prime} \simeq p^\prime_* q^{\prime *} \one_{S^\prime_{\sigma}} \simeq r^* \pi_{\sigma,*} \one_{S^\prime_{\sigma}} = r^* h(S^{\prime}_{\sigma} / S_{\phi})
\]
Since $h(S^{\prime}_{\sigma} / S_{\phi})$ belongs to the category of weight zero, smooth Tate motives over $S_{\phi}$ (again by Assumption \ref{asspmgs}.\ref{shapemor} and Rem. \ref{chowtate}), we obtain that $h(D^\prime / \tilde{S}_{\phi})$ belongs to the category of weight zero, smooth Tate motives over $\tilde{S}_{\phi}$. The motive $h(D / \tilde{S}_{\phi})$, being a direct factor of $h(D^\prime / \tilde{S}_{\phi})$, belongs to the same category as well. Now $h(D / S_{\phi})$ is isomorphic to the direct image of $h(D / \tilde{S}_{\phi})$ under the finite \'etale morphism $r$, and as a consequence, it is actually a weight zero, smooth Artin-Tate motive over $S_{\phi}$. As such, it is indeed finite dimensional (apply for example \cite[Prop. 5.8 (c)]{Wil17}). As the objects we were interested in are direct factors of objects of the form $h(D / S_{\phi})$, we conclude. 

\end{proof}

\begin{rem}\label{diff}
\begin{enumerate}[wide, labelindent=0pt, label=(\arabic*)]
\item \label{fd} If we relax point \ref{shapemor} of Assumption \ref{asspmgs} by asking that $h(S^{\prime}_{\sigma} / S_{\phi})$ be simply \emph{finite dimensional}, the proof of the above theorem carries through, with the following adjustments. First, one exploits the commutative diagram coming from the Stein factorization and invokes proper base change and \cite[Prop. 5.8 (a)]{Wil17} in order to show that $h(D / \tilde{S}_{\phi})$ is also finite dimensional. Then, one shows that the same holds for $h(D / S_{\phi})$, by applying \cite[Prop. 5.8 (c)]{Wil17}.
\item
The above theorem differs in the following way from the analogous Thm. 5.4 in \cite{Wil17}. On the one hand, in order to deal with the gluing, we are forced to be more restrictive on the choice of possible morphisms $S^\prime \rightarrow S$. In fact, we ask for regularity of the closure of the strata of $S^\prime$, whereas in \emph{loc. cit.}, it is only asked the weaker condition that for every immersion $i_{\sigma}$ of a stratum $S^{\prime}_{\sigma}$ in the closure of a stratum, the functor $i_{\sigma}^!$ send the unit object to a Tate motive. Moreover, the morphisms $\pi_{\sigma}$ in \emph{loc. cit.} can belong to a more general class than the one considered here. On the other hand, our stronger restrictions are necessary because, for a fixed $S^\prime$ which fulfils our requirements, the categories that we glue along the strata of $S^\prime$ are more general than the ones of \emph{loc. cit.}
\end{enumerate}
\end{rem}

The proof of the above theorem shows in particular that for each $\phi \in \Phi$, denoting by $Z_{\phi}$ the complement of a stratum $S_{\phi}$ in its closure $\overline{S_{\phi}}$, the category $\mbox{AT}_{\mbox{\tiny{tr}}}(\FS_{Z_{\phi}} / Z_{\phi})$ is semiprimary. So we get: 

\begin{cor}\label{intext}
Let $\pi:S^\prime \rightarrow S$ be a proper morphism of good stratifications $\FS$ and $\Phi$, satisfying Assumption \ref{asspmgs}. For each $\phi \in \Phi$, denote by $j_{\phi}: S_{\phi} \hookrightarrow \overline{S_{\phi}}$ the open immersion of a stratum in its closure. Then for each $\phi \in \Phi$, the intermediate extension functor
\[
j_{\phi,!*} : \emph{\mbox{AT}}_{\emph{\mbox{\tiny{tr}}}}(\FS_{S_{\phi}} / S_{\phi}) \hookrightarrow \emph{\mbox{AT}}_{\emph{\mbox{\tiny{tr}}}}(\FS_{\overline{S_{\phi}}} / \overline{S_{\phi}})^u
\]
is defined (as in Def. \ref{intext}). 
\end{cor}

\begin{cor}
\label{cor: decomposition}
Let $\pi:S^\prime \rightarrow S$ be a proper morphism of good stratifications $\FS$ and $\Phi$, satisfying Assumption \ref{asspmgs}. Let $M$ be an object of the category $\ATSP$. Consider the notations of the previous Corollary, and denote moreover by $i_{\phi} : \overline{S_{\phi}} \hookrightarrow S$ the closed immersion of the closure of a stratum. Then, there exist a subset $\Phi^\prime \subset \Phi$, objects $N_{\phi}$ in $\ATwr(\FS_{S_{\phi}} / S_{\phi})$, $\phi \in \Phi^\prime$, and a non-canonical isomorphism
\[
M \simeq \bigoplus\limits_{\phi \in \Phi^\prime} i_{\phi,*}j_{\phi,!*} N_{\phi}
\] 
\end{cor}
\begin{proof}
We may always suppose, for simplicity, that there is only one open stratum $U:=S_{\phi_1}$ of $\Phi$, and denote $j_{\phi_1}$ by $j:U \hookrightarrow S$.
By using  Cor. \ref{intext} and applying Prop. \ref{scind}, we know that we have an isomorphism 
\[
M \simeq j_{!*}M_U \oplus i_*N
\]
with $N$ an object of $\ATwr(\FS_Z / Z)$. By proper base change and part (2) of Thm. \ref{constmot}, and the fact that pullback along open immersions sends weight-zero objects to weight-zero objects, we know that the pullback of $N$ to any stratum $S_{\phi}$ which is open in $Z$ belongs to $\ATwr(\FS_{S_{\phi}} / S_{\phi})$. Thus, we can apply to it the functor $j_{\phi,!*}$ (defined using again Cor. \ref{intext}). The statement then follows by an iterated application of Prop. \ref{scind}. 
\end{proof}

\begin{notation}
To ease notation, in future applications we will often write $j_{\phi}$ for the immersion $i_{\phi} \circ j_{\phi}$, and write $j_{\phi,!*}$ for the functors $i_{\phi,*}j_{\phi,!*}$ appearing in the decomposition of the previous Corollary. 
\end{notation}

\subsection{Compatibility with realizations and conservativity}. 
In this paragraph, we will fix a generic point $\Spec k \rightarrow S$ of our base $S$, and we will 
make use of the two realization functors with target the categories $\Der^b_c(S_{k,\et}, \BQ_{\ell})$ and $\Der^b_c(S_k^{\an}, \BQ)$, obtained by composition with base change through $\Spec k \rightarrow S$ from the functors $\rho_{\ell}$, $\rho_B$ introduced in the \q{Notations and conventions} section. Whenever we employ one of these two functors, we will implicitly assume that the hypotheses on $S$ and $\ell$ are satisfied. These functors will still be denoted by the same symbols. 

Let us denote any of the two families of categories $\Der^b_c(S_{k,\et}, \BQ_{\ell})$ and $\Der^b_c(S_k^{\an}, \BQ)$ by the same symbol $\Der^b_c(S_k)$. Both families of categories are equipped with a perverse $t$-structure, whose heart (the corresponding category of \emph{perverse sheaves}) will be denoted $\Perv_c(S_k)$ in both cases. We will then denote by
\[
H^m:\Der^b_c(S_k) \rightarrow \Perv_c(S_k)
\]
the perverse cohomology functors, and if $j:U \hookrightarrow S$ is an open immersion, by 
\[
j_{!*} : \Perv_c(U_k) \rightarrow \Perv_c(S_k) 
\]
the \emph{intermediate extension} of perverse sheaves (\cite[Déf. 1.4.22]{BBD82}). The composition of the collection of the perverse cohomology functors with one of the realization functors will be called the corresponding \emph{perverse cohomological realization} functor.

The following result gives the compatibility of the functor of Def. \ref{intext} (when available) with the realization functors:


\begin{thm}\label{real}
Let $\pi:S^{\prime} \rightarrow S$ be a proper morphism of good stratifications $\FS$ and $\Phi$, satisfying Assumption \ref{asspmgs}. Denote by $\rho$ any of the two realization functors $\rho_{\ell}$ or $\rho_B$. For each $\phi \in \Phi$, denote by $j_{\phi}: S_{\phi} \hookrightarrow \overline{S_{\phi}}$ the open immersion of a stratum in its closure. Then:
\begin{enumerate}[wide, labelindent=0pt, label=(\arabic*)]
\item for any $\phi \in \Phi$, for any integer $m$, the restriction of the composition 
\[
H^m \circ \rho : \DM_c(\overline{S_{\phi}},\BQ) \rightarrow \Perv_c(\overline{S_{\phi}}_k)
\]
to $\emph{\mbox{AT}}_{\emph{\mbox{\tiny{tr}}}}(\FS_{\overline{S_{\phi}}} / \overline{S_{\phi}})$ factors over $\emph{\mbox{AT}}_{\emph{\mbox{\tiny{tr}}}}(\FS_{\overline{S_{\phi}}} / \overline{S_{\phi}})^u$;
\item for any $\phi \in \Phi$, for any integer $m$, the diagram 
\begin{center}
\begin{tabular}{c}
\xymatrix{
\emph{\mbox{AT}}_{\emph{\mbox{\tiny{tr}}}}(\FS_{S_{\phi}} / S_{\phi}) \ar[r]^{j_{\phi,!*}} \ar[d]_{H^m \circ \rho} & 
\emph{\mbox{AT}}_{\emph{\mbox{\tiny{tr}}}}(\FS_{\overline{S_{\phi}}} / \overline{S_{\phi}})^u \ar[d]^{H^m \circ \rho} \\
\Perv_c(S_{\phi,_k}) \ar[r]^{j_{!*}} & \Perv_c(\overline{S_{\phi}}_k)
}
\end{tabular}
\end{center}
commutes. 
\end{enumerate}
\end{thm}
\begin{proof}
The same proof of \cite[Thm. 7.2]{Wil17} applies. Indeed, the only ingredient occurring in that proof, which has to be generalized, is \emph{op. cit.}, Cor. 7.13: we need to obtain an analogous statement in our situation, i.e. we need to show that the radical (cfr. Def. \ref{semipr}) of our categories $\mbox{AT}_{\mbox{\tiny{tr}}}(\FS_{\overline{S_{\phi}}} / \overline{S_{\phi}})$ is mapped to zero under the perverse cohomological realization. The proof of Thm. \ref{artatesp} tells us that over each stratum $S_{\phi}$ of $S$, the category $\mbox{AT}_{\mbox{\tiny{tr}}}(\FS_{S_{\phi}} / S_{\phi})$ is actually contained in the category $\mbox{AT}(S_{\phi})$ of weight zero, smooth Artin-Tate motives over $S_{\phi}$. Then, the same strategy of proof of \emph{op. cit.}, Cor. 7.13, based on \emph{op. cit.}, Thm. 7.12, shows that we can obtain the desired statement as soon as we prove the following claim: for any couple of finite \'etale morphisms 
\[
q_1:D_1 \rightarrow X, \ q_2:D_2 \rightarrow X
\]
the ideal 
\[
\rad_{\Chow(X)}(h_X(D_1), h_{X}(D_2))
\]
consists of morphisms which are mapped to zero under the perverse cohomological realization. Now, by the same reasoning as in the proof of \emph{loc. cit.}, Thm. 7.12, we may assume that $X$ is the spectrum of a field $k$. But then, $h_X(D_1)$ and $h_{X}(D_2)$ are Artin motives over $k$, i.e. objects of a full \emph{semisimple} subcategory of $\Chow(X)=\Chow(k)$. Hence, the radical in question is zero. 
\end{proof}

%

We end this section by discussing the conservativity of the restriction of the realization functors to the categories $\DMATSP$. 

\begin{thm} \label{cons}
Let $\pi:S^{\prime} \rightarrow S$ be a proper morphism of good stratifications $\FS$ and $\Phi$, satisfying Assumption \ref{asspmgs}. Denote by $\rho$ any of the two realization functors $\rho_{\ell}$ or $\rho_B$. Then, the restriction of $\rho$ to $\DMATSP$,
\[
\rho: \DMATSP \rightarrow \Der^b_c(S_k)
\]
is conservative.
\end{thm}
\begin{proof}
We will adopt the strategy of proof \cite[Thm. 4.3]{Wil18}, which adapts to our setting and shows that conservativity of both realizations ($\ell$-adic and Betti) can be deduced if the following properties hold\footnote{The proof of \emph{loc. cit.} takes as starting point \emph{op. cit.}, Thm. 2.10, which shows that conservativity follows once we add to these hypotheses an additional one: \emph{strictness} of any morphism in the image of the perverse cohomological realization, with respect to the \emph{weight filtration} of the latter functor. Only the $\ell$-adic realization is known to satisfy this last assumption. Building on this, the proof of Thm. 4.3 of \emph{loc. cit.} then shows how to obtain conservativity for the Betti realization, too.}: (1) the weight structure on $\DMATSP$ is bounded, (2) its heart $\ATSP$ is semi-primary and pseudo-Abelian, (3) the restriction of $\rho$ to $\ATSP$ maps the radical to zero, and (4) zero is the only object of $\ATSP$ mapped to zero by $\rho$. The first two properties have been verified before (Lemma \ref{ws} and Thm. \ref{artatesp}), and property (3) has been seen to hold in the course of the proof of Thm. \ref{real}. With this in our hands, we can imitate step by step the proof of \emph{op. cit.}, Thm. 4.2 to show that property (4) is also verified and to conclude. In fact, to argue as in \emph{loc. cit.} we only need the existence of the intermediate extension functor defined in Cor. \ref{intext}, the fact that objects in $\ATSP$ decompose as direct sums of intermediate extensions (Cor. \ref{cor: decomposition}), and the finite dimensional nature of the objects of the categories $\mbox{AT}_{\mbox{\tiny{tr}}}(\FS_{S_{\phi}} / S_{\phi})$ (proof of Thm. \ref{artatesp}). 
\end{proof}

\section{Motivic decompositions of projective quadric bundles}

\subsection{Corti-Hanamura decomposition of general quadrics}

\begin{num}
In this section we work over a field $k$ of characteristic zero. Let $V$ be a $k$--vector space of dimension $n$.
 A quadratic form $Q\in S^2V^{\vee}$ can be viewed as a symmetric map $Q:V\to V^{\vee}$. Its radical is the subspace $\Rad(Q) = \ker(Q)$.
 It is known from the theory of quadratic forms that $Q$ descends to a symmetric bilinear form $\Qbar$ on the quotient space $\Vbar = V/\Rad(Q)$ (cf. \cite[2.6]{Cassels}).
 We say that $Q$ has \emph{corank $r$} if $\dim(\Rad(Q))=r$. The zero locus of $Q$ defines a quadric $X\subset\PP(V)$ which is smooth if and only if $r=0$. If $r>0$ its singular locus is the linear subspace $\Lambda =\PP(\Rad(Q))$, and $X$ is a cone with vertex $\Lambda$ over the smooth quadric $\Xbar = V(\Qbar)\subset\PP(\Vbar)\cong\PP^{n-r-1}_k$. 

\medskip
Let 
$$
\Delta_k = \{Q\in S^2V^{\vee}|\dim (\ker Q)\ge k\}  
$$
be the space of quadratic forms of corank $\ge k$. 
The standard desingularization of $\Delta_k$ is
$$
\widehat{\Delta}_k = \{(Q,F)\in S^2V^{\vee}\times G(k,V)| F\subset\ker Q\}.
$$
For our purposes it is convenient to use a slightly different construction. We denote by $\Fl(V) = \Fl(1,\ldots,n-1,V)$ the variety of complete flags 
$$
F_1\subset ... \subset F_{n-1}\subset F_n=V
$$
where $\dim(F_i) = i$. Define
$$
\widetilde{\Delta}_k = \{(Q,F_{\bullet})\in S^2V^{\vee}\times\Fl(V)|F_k\subset\ker Q\}.
$$
The fiber of the projection map $p_2:\widetilde{\Delta}_k\to\Fl(V)$ over $F_{\bullet}$ is the vector space $H^0(\PP(V),{\cal I}_{\PP(F_k)}(2))$. Hence $\widetilde{\Delta}_k$ is smooth. The advantage over the previous construction is that we have inclusions
$$
\widetilde{\Delta}_{k+1}\subset\widetilde{\Delta}_k\subset\Fl(V)
$$
for all $k$.
\end{num}

\begin{num}
We now discuss the analogue of these constructions in the relative case. Let $S$ be a quasi--projective scheme over $k$. Let $E$ be a rank $n$ vector bundle over $S$, and let $L$ be a line bundle over $S$. A quadratic form on $E$ with values in $L$ is a global section $q\in H^0(S,S^2E^{\vee}\otimes L)$, or equivalently a symmetric homomorphism
$$
q:E\to E^{\vee}\otimes L.
$$
Let $\rho:\PP(E)\to S$ be the associated projective bundle with tautological line bundle $\xi_E = \CO_{\PP(E)}(1)$. Using the isomorphism $$H^0(S,S^2E^{\vee}\otimes L)\cong H^0(\PP(E),\xi_E^2\otimes\rho^*L)$$ we can identify $q$ with a global section (still denoted by $q$) of  $\xi_E^2\otimes\rho^*L$. The associated quadric bundle is $X = V(q)\subset\PP(E)$. The fiber of $f:X\to S$ over $s\in S$ is the zero locus of $q(s)\in S^2 E_s^{\vee}\otimes L_s$.  

\medskip
We write
$$
\Delta_i(q) = \{s\in S\vert \corank q(s)\ge i\},\ \ U_i = \Delta_i(q)\setminus\Delta_{i+1}(q).
$$
The restriction of $q$ to $U_i$ defines a homomorphism of vector bundles 
$$
q_i:E_i\to E_i^{\vee}\otimes L_i
$$
whose kernel $F_i = \ker(q_i)$ is a subbundle of $E_i$ of rank $i$. As before, the quadratic form $q_i$ descends to a quadratic form $\overline{q}_i$ on the quotient $\Ebar_i = E_i/F_i$. Geometrically this means that
the subscheme $X_i = f^{-1}(U_i)$ is a relative cone over $\Xbar_i = V(\qbar_i)\subset\PP(\Ebar_i)$ with vertex $\PP(F_i)$ i.e., for every $s\in U_i$ the quadric $X_s = f^{-1}(s)$ is a cone with vertex $\PP(F_{i,s})$ over $\Xbar_s\subset\PP(\Ebar_{i,s})$. 




Let $\Fl(E)$ be the bundle of complete flags in $E$ with projection map $\pi:\Fl(E)\to S$. The vector bundle $\pi^*E$ has a flag of universal subbundles $S_i$ of rank $i$, $i=1,\ldots,n-1$.  The composition of  
$$
\pi^*q:\pi^*E\to\pi^*E^{\vee}\otimes\pi^*L
$$
and the inclusion $\lambda_i:S_i\to\pi^*E$ defines 
$$
\widetilde{q}_i = \pi^*(q)\circ\lambda_i:S_i\to\pi^*E^{\vee}\otimes\pi^*L.
$$
Define 
$$
{\widetilde\Delta}_i(q) = V(\widetilde{q}_i)\subset\Fl(E).
$$

\end{num}
\begin{prop}
\label{prop: regular}
If 
$E^{\vee}\otimes E^{\vee}\otimes L$ is generated by global sections and $q\in H^0(X,S^2E^{\vee}\otimes L)$ is general, then
\begin{enumerate}
\item $\Delta_i(q)$ is empty or has the expected codimension $\binom{i+1}2$ 
 and $\Sing(\Delta_i(q)) = \Delta_{i+1}(q)$.
\item $\widetilde{\Delta}_i(q)$ is smooth. 
\end{enumerate}
\end{prop}

\begin{proof}


Note that if $E^{\vee}\otimes E^{\vee}\otimes L$ is generated by global sections, then the bundles $S^2E^{\vee}\otimes L$, $\pi^*(E^{\vee}\otimes E^{\vee}\otimes L)$ and $S_i^{\vee}\otimes\pi^*E^{\vee}\otimes\pi^*L$ are globally generated.
Part 1 is proved by adapting the argument of \cite[4.1]{Banica} to the symmetric case; cf. \cite[2.17]{Ottaviani}. Part 2 follows from Bertini's theorem.  
\end{proof}

\begin{df} \label{def:regular}
We say that $X\to S$ is a {\em regular} quadric bundle if it satisfies the conclusions of the Proposition \ref{prop: regular}.  
\end{df}

\medskip
A quadric bundle $f:X\to S$ admits a natural stratification by corank. Write $U_i = \Delta_i(q)\setminus\Delta_{i+1}(q)$, $X_i = f^{-1}(U_i)$. As the stratification ${\Phi} = \{U_i\}_{i\in I}$ does not verify Assumption \ref{assreg}, we have to pass to a suitable base change and verify Assumption \ref{asspmgs}. Define $S' = \Fl(E)$ and consider the stratification $\mathfrak S$ given by $U_i'=\widetilde{\Delta}_i(q)\setminus\widetilde{\Delta}_{i+1}(q)$. Write $X' = X\times_S S'$ and $X_i' = X_i\times_{U_i} U_i'$.

\begin{lm}
\label{lemma: Assumption2}
Let $X\to S$ be a regular quadric bundle. Then the stratification $\FS = \{U'_i\}_{i\in I}$ satisfies Assumption \ref{asspmgs}.
\end{lm}

\begin{proof}
Proposition \ref{prop: regular} implies that the stratification $\FS$ satisfies conditions (1) and (2) of Assumption \ref{asspmgs}. 
The definition of $\widetilde{\Delta}_i(q)$ shows that the fiber of $\pi_i:\widetilde{\Delta}_i(q)\to\Delta_i(q)$ over $s\in\Delta_i(q)$ is
$$
\{(W_{\bullet})\in\Fl(E_s)|W_i\subset\ker q(s)\}.
$$
If $s\in U_i$ then $\ker(q(s))$ has dimension $i$, hence the fiber of the induced map $U_i'\to U_i$ over $s\in U_i$ is
$$
\{(W_{\bullet})\in\Fl(E_s)|W_i = \ker(q(s)\}\cong\Fl(\ker(q(s))\times\Fl(i+1,\ldots,n,E_s).
$$
Over $U_i$ we have an injective homomorphism of flag bundles
$$
\Fl(F_i)\times\Fl(i+1,\ldots,n;E_i)\to\Fl(E_i)
$$
whose image is $U'_i$. Hence $U_i'\to U_i$ is a relative homogeneous space and $h_{U_i}(U'_i)$ is a relative Tate motive. This implies that the stratification $\FS$ also satisfies condition (3) of Assumption \ref{asspmgs}.
\end{proof}

\begin{thm} \label{quadric_AT}
Let $X\to S$ be a regular quadric bundle. Then $h_S(X)\in\ATSP$.
\end{thm}

\begin{proof} 
We consider the stratifications $\Phi$ and $\FS$ introduced above and write $X_i = f^{-1}(U_i)$, $X'_i = X_i\times_{U_i}U'_i$.
As we have seen before, $X_i$ is a relative cone over $\overline{X}_i$ with vertex $\PP(F_i)$. The complement $V_i = X_i\setminus\PP(F_i)$ is a locally trivial fiber bundle over 
$\overline{X}_i$ with affine fibers. 
\footnote{In fact, if  $\pi_i:\PP({E}_i)\to U_i$ is the projection map and if one chooses local trivialisations $F_i|_U\cong W\otimes\CO_U$, $E_i|_U\cong V\otimes\CO_U$ over an open subset $U\subset U_i$ then $V_i\cap\pi_i^{-1}(U)$ is the total space of the vector bundle $\CO_{U}(-1)\oplus W\otimes\CO_{U}$; cf. \cite[9.3.2]{EH}}
A similar result holds after base change via $\pi_i:U'_i\to U_i$: $X'_i$ contains the projective bundle  $\PP(F'_i)$, where $F'_i = \pi_i^*(F_i)$, and the projection of the complement $V'_i = X'_i\setminus\PP(F'_i)$ to $U'_i$ factors as $V'_i\xrightarrow{\rho_i}\overline{X}'_i\xrightarrow{\sigma_i}U'_i$ where $\rho_i$ is a locally trivial fiber bundle with affine fibers and $\sigma_i$ is a smooth quadric bundle. Hence $h^c_{U'_i}(V'_i)\cong h_{U'_i}(\overline{X}'_i)(-i)[-2i]$. As we have seen in Corollary \ref{cor:sm_quadrics},  the motive 
$h_{U'_i}(\overline{X}'_i)$ is either a relative Tate motive or a relative Artin-Tate motive associated to a double \'etale covering $Z(\overline{X}'_i)\to U'_i$. The latter case arises if $n-i$ is even, say $n-i=2m$. In this case the double covering $Z(\overline{X}'_i)\to U'_i$ comes from the Stein factorisation of the relative Fano scheme of $m$--planes 
$F_m(\overline{X}'_i/U'_i)\to U'_i$. Since we work in characteristic zero,  the map $Z(\overline{X}'_i)\to U'_i$ is tamely ramified by Remark \ref{rem: tr}.
As $h_{U'_i}(\PP(F'_i))$ is a relative Tate motive, the localisation triangle 
$$
h^c_{U'_i}(V'_i)\to h_{U'_i}(X'_i)\to h_{U'_i}(\PP(F'_i))\to h^c_{U'_i}(V'_i)[1]
$$
then shows that $h_{U'_i}(X'_i)\in \DMAwT(U'_i, \BQ)$ for all $i$. As the map $h_{U'_i}(\PP(F'_i))\to h^c_{U'_i}(V'_i)[1]$ is zero for weight reasons, it follows that $h_{U'_i}(X'_i)$ has weight zero. Hence $h_{S'}(X')\in \DMATS$ thanks to part (2) of Thm. \ref{constmot} and has weight zero by the gluing property of motivic weights. 

As $S' = \Fl(E)\xrightarrow{\pi} S$ is a relative homogeneous space we have
$$
\pi_*\un_{S'}\cong\un_S\oplus(\bigoplus_{i:n_i>0}\un_S(-n_i)[-2n_i]).
$$
By the projection formula we obtain
$$
\pi_*h_{S'}(X') = \pi_*\pi^*h_S(X)\cong h_S(X)\oplus(\bigoplus_{i:n_i>0}h_S(X)(-n_i)[-2n_i]).
$$
Hence $h_S(X)\in\DMATSP$.  Since proper morphisms respect weights and $h_{S'}(X')$ is of weight zero, $\pi_*h_{S'}(X')$ is of weight zero, as well as any of its direct factors. So $h_S(X) \in\ATSP$.
\end{proof}

\begin{cor} \label{quadricCH}
A regular quadric bundle $X\to S$ admits a CH-decomposition (see Def. \ref{df:CH}).
\end{cor}

\begin{proof}
By Corollary \ref{cor: decomposition} every object $M$ of the category $\ATSP$ admits a decomposition
$$
M \simeq \bigoplus\limits_{\phi \in \Phi^\prime} i_{\phi,*}j_{\phi,!*} N_{\phi}
$$
with $N_{\phi}$ in $\ATwr(\FS_{S_{\phi}}/S_{\phi})$.  The proof of Theorem \ref{artatesp} shows that $N_{\phi}\in\DMAT(S_{\phi},\QQ)$. Using Theorem \ref{thm:Artin-Tate} we obtain a decomposition
$$
N_{\phi} \simeq \bigoplus_{i,\phi \in I_{\phi}} \rho_!(V_{i,\phi})(n_{i,\phi})[2n_{i,\phi}]
$$
with $V_{i,\phi}$ a simple Artin representation of $\pi_1(S_{\phi})$.  Taking $M = h_S(X)$ we get
$$
h_S(X)\simeq\bigoplus_{i,\phi} i_{\phi,*}j_{\phi,!*}(\rho_!(V_{i,\phi})(n_{i,\phi}))[2n_{i,\phi}]
$$
which is a decomposition of $h_S(X)$ into simple objects of weight zero. Applying the realisation functor $\rho_B$ and using Thm. \ref{real}, we obtain a decomposition of $Rf_*\QQ_X$ into a sum of simple objects in $D^b_c(S,\QQ)$.  The decomposition theorem of Beilinson-Bernstein-Deligne gives isomorphisms
$$
Rf_*\QQ_X\simeq\bigoplus_k {^p}R^kf_*\QQ_X[-k]\simeq\bigoplus_{k,\lambda} i_{\lambda,*}j_{\lambda,!*}(L_{k,\lambda})[-k]
$$
where $L_{k,\lambda}$ is a local system on $S_{\lambda}$. Since the simple objects appearing in this decomposition are unique,  we conclude that they are realisations of Chow motives. Hence $X\to S$ admits a Corti-Hanamura decomposition.
\end{proof}

\begin{rem}
Our proof gives a more precise statement: $h_S(X)$ is a direct sum of motivic intermediate extensions of Artin-Tate motives of degree at most 2, hence $Rf_*\QQ_X$ decomposes as a sum of intersection complexes of  local systems whose monodromy is either trivial or $\ZZ/2\ZZ$. This has been observed in the example of regular conic bundles over a surface 
\cite{NS}. In this case $$Rf_*\QQ_X\simeq\QQ_S\oplus\QQ_S[-2]\oplus i_*L[-1]$$ where $L$ is a local system on the smooth discriminant curve $\Delta\subset S$. The underlying motivic decomposition is 
$$
h_S(X)\simeq \un_S\oplus \un_S(-1)[-2]\oplus {\rm Prym}(\widetilde{\Delta}/\Delta)(-1)[-2]
$$
where ${\rm Prym}(\widetilde{\Delta}/\Delta)$ is the Prym motive, an Artin-Tate motive of degree 2.
\end{rem}

\subsection{Corti-Hanamura decomposition of quadric bundles with normal crossing discriminant}

\begin{num}\label{num:ncd}
The aim of this final section is to show that given the previous results, it is also quite immediate to get a Corti-Hanamura decomposition for a class of quadric bundles $f:X \rightarrow S$ over a field $k$ of characteristic zero, which don't satisfy anymore the regularity assumption. We will instead suppose that the discriminant locus $i: D \hookrightarrow S$ of $f$ is a \emph{normal crossing divisor}. To fix our conventions, this means that $D$ is a finite union of irreducible components
\[
D = \bigcup\limits_{i \in I} D_i
\] 
such that for each finite subset $J \subseteq I$, the underlying reduced scheme of the scheme $D_{J}:= \bigcap\limits_{j \in J} D_j$ is regular. 

Fix now an ordering $D_{J_1}, \dots, D_{J_r}$ of the set $\{ D_{J} \}_{J \subseteq I}$, in such a way that the function $\dim D_{J_i}$ is (non-strictly) increasing with respect to $i$, and pose the following:
\end{num}

\begin{df} \label{strat_sncd}
The stratification $\FS$ of $S$ 
is the one obtained by choosing, for each $\sigma \in \FS:=\{1, \dots, r \}$,
\[
S_{\sigma}:=D_{J_{\sigma}} \setminus ( \bigcup\limits_{k=1, \cdots, \sigma-1} D_{J_k})
\]
\end{df}
Note that the above defined $\FS$ is a well-defined stratification (i.e., the closure of each stratum is a union of strata). Then, it is clear from the definitions that we have: 

\begin{lm}
The stratification $\FS$ satisfies Assumption \ref{assreg}.  
\end{lm}

Hence, we can define the full, triangulated subcategory $\DMATS$ of $\DM_c(S)$ as in Thm. \ref{constmot}. 

\begin{thm}
Let $X \rightarrow S$ be a quadric bundle whose discriminant is a divisor with normal crossings. Then, the motive $h_S(X)$ belongs to the category $\ATS$. 
\end{thm}
\begin{proof}
For each $\sigma \in \FS$, the fibers of $X_{S_{\sigma}}$ over $S_{\sigma}$ have constant corank. Hence, the same computation as in the proof of Thm. \ref{quadric_AT} shows that $h_{S_{\sigma}}(X_{S_{\sigma}})$ belongs to $\DMAwT(S_{\sigma}, \BQ)$ for each $\sigma \in \FS$, and gives the conclusion. 
\end{proof}

Finally, the proof of Corollary \ref{quadricCH} can be repeated in order to show: 

\begin{cor} \label{quadricCHsncd}
A quadric bundle $X\to S$ whose discriminant is a divisor with normal crossings admits a CH-decomposition.
\end{cor}

We present an example where several nonconstant intersection motives appear in the decomposition. First we need a Lemma.

\begin{lm}
Let $\pi:S'\to S$ be a finite morphism. Write $U'=\pi^{-1}(U)$, $\pi'=\pi|_{U'}$ and let $j:U\to S$, $j':U'\to S$ be the inclusions.
If $F$ is a perverse sheaf on $U'$ then
$$
j_{!*}(\pi'_*F)\simeq \pi_*(j'_{!*}(F)).
$$
\end{lm}

\begin{proof}
As $\pi$ and $\pi'$ are finite morphisms, the functors $\pi_*$ and $\pi'_*$ are t-exact for the perverse t-structure \cite[Cor. 2.2.6 (i)]{BBD82}. The result then follows from the isomorphisms
$$
\pi_*\circ j'_*\simeq j_*\circ\pi'_*, \ \
\pi_*\circ j'_!\simeq j_!\circ\pi'_*
$$
using [loc.cit., 1.3.17 (iv) and 1.4.16].
\end{proof}
\begin{ex} \label{ex_manyIC}
Let $r$ be a nonnegative integer. Given a (r+1)x(2n+3) matrix $A=(a_{ij})$ we define quadrics  
$$
Q_i = \{x\in\PP^{2n+2}| \sum_j a_{ij}x_j^2 = 0\}\subset\PP^{2n+2},\ \ i=0,\ldots,r.
$$
The quadric bundle associated to the linear system of quadrics spanned by $Q_0,\ldots,Q_r$ is $f:X\to\PP^r$ with 
$$
X = \{(x,\lambda)\in\PP^{2n+2}\times\PP^r|\sum_{i,j} \lambda_i a_{ij}x_j^2=0\}.
$$
The fiber $Q_{\lambda}$ of $f$ over $\lambda = (\lambda_0:\ldots:\lambda_r)\in\PP^r$ is the zero locus of
$$
q_{\lambda} = \sum_j(\sum_i a_{ij}\lambda_i)x_j^2.
$$
If $A$ is a generic matrix, the discriminant locus of $f$ is a simple normal crossing divisor which is the union of $2n+3$ hyperplanes: $\Delta = D_0\cup\ldots\cup D_{2n+2}$ where
$$
D_j = \{\lambda\in\PP^r| \sum_i {a_{ij}}\lambda_i = 0\},\ \ j=0,\ldots,2n+2.
$$
By the previous Corollary $f:X\to S$ admits a CH-decomposition.
We shall show that it contains several nonconstant intersection motives by looking at the topological realisation. Write $d_X= \dim X = 2n+r+1$.
By the decomposition theorem \footnote{For ease of comparison with the existing literature we have adopted the usual convention of shifting by $d_X$.}
$$
Rf_*(\QQ_X[d_X])\simeq\bigoplus_{i=-2n-1}^{2n+1}{^p R}^i f_*(\QQ_X[d_X])[-i]
$$
and the perverse direct images ${^p R}^i f_*(\QQ_X[d_X])$ are direct sums of intermediate extensions. By the perverse versions of the Lefschetz hyperplane theorem and the hard Lefschetz theorem we have
$$
{^p R}^i f_*(\QQ_X[d_X]) = \left\{ 
\begin{array}{cc}
\QQ[r] & i\ne 0, i \hbox{\rm\ even\ } \\
0 & i\ne 0, i\hbox{\rm\ odd\ }.
\end{array}
\right.
$$
hence
$$
Rf_*\QQ_X[d_X]\simeq\bigoplus_{k=-n}^{n}\QQ[r-2k-1]\oplus {^p R}^0 f_*(\QQ_X[d_X]).
$$
Write 
$$
\Delta_i = \cup_{j\ne i} D_i\cap D_j,\ \ 
S_i = D_i\setminus\Delta_i.
$$
If $\lambda\in S_i$ then $Q_{\lambda}$ has an ordinary double point $p_{\lambda}$ and $Q_{\lambda}$ is a cone
 with vertex $p_{\lambda}$ over a smooth quadric $\overline{Q_{\lambda}}\subset\PP^{2n+1}$.
 Hence 
$$
R^{2n+2}f_*\QQ|_{S_i} = \QQ\oplus L_i
$$
where $L_i$ is a nontrivial rank one local system on $S_i$ with monodromy 
$\ZZ/2\ZZ$. The decomposition theorem then shows that
$$
R^{2n+2}f_*\QQ|_{S_i} = {\cal H}^{-r+1}(Rf_*\QQ_X[d_X]|_{S_i}) = \QQ[r-1]\oplus{\cal H}^{-r+1}({^p R}^0 f_*(\QQ_X[d_X]))
$$
for all $i$. Thus, writing $j_{\alpha}:S_{\alpha}\to\PP^r$ for the inclusion maps, 
$$
{^p R}^0 f_*(\QQ_X[d_X]) \simeq \bigoplus_{\alpha = 0}^r ({j_{\alpha}})_{!*}(L_{\alpha}) \oplus R
$$
where $R$ is a perverse sheaf supported in codimension $\ge 2$. Hence the CH-decomposition of $X\to S$ is of the form
$$
h_S(X)\simeq\bigoplus_{k=0}^n \un(-k)[-2k]\oplus\bigoplus_{i=0}^r M_i \oplus M
$$
where $\rho(M)=R$ and $\rho(M_i) = j_{!*}(L_i)$ for all $i$. 

\medskip 
\begin{rem}
In the case $r=2$ there is a direct construction of the motive $M_i$ underlying the intermediate extension of $L_i$. In this case $D_i\cong\PP^1$ and $\Delta_i\subset D_i$ is a set of $2n+2$ points. Let $D'_i\to D_i$ be the double covering ramified along $\Delta_i$. Then $D'_i$ is a smooth hyperelliptic curve of genus $n$. Applying the previous lemma to the diagram
\begin{center}
\begin{tabular}{c}
\xymatrix{
S_i'\ar[d]^{\pi'} \ar[r]^{j'} &   D'_i \ar[d]^{\pi} \\ 
S_i \ar[r]^{j} &   D_i}
\end{tabular}
\end{center}
we obtain 
$
\pi_*(j'_{!*}\QQ)\simeq \QQ\oplus j_{!*}(L_i)
$
since $\pi'_*\QQ\simeq\QQ\oplus L_i$. 
 
If we denote $\tau$ the hyperelliptic involution on $D'_i$, the intermediate extension $j_{!*}(L_i)$ is the realisation of the Prym motive $M_i = (D'_i,\frac{1}{2}(\tau^*+{\rm id}))$. 

It is not clear whether there exists a similar description of the motives $M_i$ if $r>2$. In this case the variety $D'_i$ is singular, hence $M_i$ should be constructed as a suitable direct factor of the "intersection motive" of $D'_i$.     
\end{rem}

\begin{rem}
If we take $n=0$ and $r=2$ in the example above (conic bundles over $\PP^2$)  the result follows from \cite{NS} (The result in [loc.cit.] is more general and applies to conic bundle over a surface with a reducible discriminant locus that is not necessarily a simple normal crossing divisor.)  
\end{rem}

\end{ex}

\bibliographystyle{amsalpha}
\bibliography{decomposition}

\end{document}